\documentclass[11pt]{amsart}
\usepackage{a4wide,enumerate,enumitem,amsmath,color}
\usepackage[utf8]{inputenc}
\usepackage[scale = .8]{geometry}
\usepackage{amssymb}
\usepackage{amsthm}
\usepackage{graphicx}
\usepackage{subcaption}
\usepackage[noadjust]{cite}
\usepackage{bigints}
\usepackage{breqn}
\usepackage{tikz}
\usepackage{mathtools}
\usepackage{hyperref}
\usepackage{cite}
\usepackage{dsfont}
\usepackage {float}

\newtheorem{theorem}{Theorem}

\newtheorem{lemma}{Lemma}
\newtheorem{proposition}{Proposition}
\newtheorem{remark}{Remark}

\newcommand{\K}{\mathcal{K}}
\newcommand{\dd}{\mathrm{d}}
\newcommand{\pa}{\partial}
\newcommand{\F}{\mathcal{F}}
\newcommand{\R}{\mathbb{R}}
\newcommand{\Z}{\mathbb{Z}}

\newcommand{\I}{\mathcal{I}_N}
\newcommand{\N}{\mathbb{N}}
\newcommand{\eps}{\varepsilon}

\begin{document}

\title[Convergence of a scheme for a periodic nonlocal eikonal equation]{Convergence of a semi-explicit scheme for a one dimensional periodic nonlocal eikonal equation modeling dislocation dynamics}

\author[D. Al Zareef]{Diana Al Zareef}
\address{Universit\'e de Technologie de Compi\`egne, LMAC, 60200 Compi\`egne, France}
\email{diana.al-zareef@utc.fr}

\author[A. El Hajj]{Ahmad El Hajj}
\address{Universit\'e de Technologie de Compi\`egne, LMAC, 60200 Compi\`egne, France}
\email{elhajjah@utc.fr}

\author[H. Ibrahim]{Hassan Ibrahim}
\address{Faculty of Science I, Mathematics Department, Lebanese University, Hadath, Lebanon}
\email{hsnibrahim81@gmail.com}

\author[A. Zurek]{Antoine Zurek}
\address{CNRS – Université de Montréal CRM – CNRS, Montréal, Canada and Universit\'e de Technologie de Compi\`egne, LMAC, 60200 Compi\`egne, France}
\email{antoine.zurek@utc.fr}

\date{\today}

\thanks{The authors thank the anonymous referees for comments and corrections that helped us to improve this paper substantially.} 

\begin{abstract}
In this paper, we derive a periodic model from a one dimensional nonlocal eikonal equation set on the full space modeling dislocation dynamics. Thanks to a gradient entropy estimate, we show that this periodic model converges toward the initial one when the period goes to infinity. Moreover, we design a semi-explicit numerical scheme for the periodic model that we introduce. We show the well-posedness of the scheme and a discrete gradient entropy inequality. We also prove the convergence of the scheme and we present some numerical experiments.
\end{abstract}

%\paragraph{Keywords:}
%\keywords{}

%\paragraph{AMS classification:}
%\subjclass[2000]{}

\maketitle

\tableofcontents

%%%%%%%%%%%%%%%%%%%%%%%%%%%%%%%%%%%%%%%%%%%%%%%%%%%%%%%%%%%%%%%%%%%%%%%%%%%%%%%

\section{Introduction}
\subsection{Physical motivation}
In this paper, we are interested in the numerical study of a phase field model introduced by Rodney, Le Bouar and Finel in~\cite{22} describing the dynamics of dislocation curves in crystal materials. Dislocations are linear defects that glide in crystallographic slip planes and  whose motion is responsible for the plastic deformation of metals (see \cite{19} for a physical description of dislocations). A given dislocation moves under the action of the so-called Peach-Koehler force computed from linear elasticity equations, where the internal self-stress generated by the dislocation curve represents one of its principal components. We are mainly concerned with the viscoplastic behavior of materials and therefore we neglect all other stresses (exterior to the dislocation curve) that may also contribute to the Peach-Koehler force. We consider a rather simple geometric setting where the dislocations are parallel lines moving in the same two-dimensional $(xy)$ plane embedded in a three dimensional elastic crystal. A cross-sectional view then reduces the geometric study into a one-dimensional framework. To be more specific, we study a one-dimensional model given by the following non-local eikonal equation describing dislocations dynamics
\begin{align}\label{non periodic equation0}
    \begin{cases}
  \partial_t v(x,t) = \left( \K(\cdot) \star v(\cdot,t)(x)\right)
  \left|\partial_xv(x,t)\right|& \mbox{in } \mathbb{R} \times (0, T), \\
  v(x,0) = v_0(x) & \mbox{in } \mathbb{R},
\end{cases}
\end{align}
where $T>0$, the solution $v$ is a real-valued function,  $\partial_t v$, $\partial_x v$ stand respectively for its time and space derivatives and $\K$ is a kernel obtained by the resolution of the equations of linear elasticity. Here the dislocations move with the non-local velocity
$$\K(\cdot) \star v(\cdot,t)(x) = \int_{\R} \K(x-z)v(z,t)dz$$
created by the self-force. In physical terms, it corresponds to the interior elastic stress created by all the dislocations in the material. While the scalar function $v(\cdot,t)$ represents the plastic distortions due to the presence of dislocations, it is chosen so as to capture the dislocation points at a time $t$ (one might roughly think of $v(\cdot,t)$ as the evolution of a sum of weighted Heaviside functions).  

\subsection{Brief review of the literature}
The solution of~\eqref{non periodic equation0} is considered in the sense of continuous viscosity solutions. The theory of viscosity solutions was first introduced by Crandall and Lions in 1980 (for a general overview, see for example the book of Barles~\cite{5}). Many authors used this theory for the resolution of non-local eikonal equation modeling dislocation dynamics. Alvarez et al.~\cite{311} proved short time existence and uniqueness results of discontinuous viscosity solutions under the assumption that dislocations have the shape of Lipschitz graphs. 

Under the restrictive assumption of nonnegative velocity, we refer to~\cite{211, 1720} for a study of long time existence and uniqueness of discontinuous solutions. These  results were recently extended by Barles et al.~\cite{8}, using a new approach that allows to relax some of the assumptions of~\cite{211, 1720}. 

In the case where velocity is allowed to change sign, Barles et al.~\cite{7} obtained a global existence result of weak $L^1$-viscosity solutions. The authors also investigated the relation of this type of solutions with the classical one, as well as the uniqueness in several situations. However, a general global existence and uniqueness result was obtained in~\cite{8} under general and simplified assumptions. Recently,~\cite{10} and~\cite{11} established global existence results for weak discontinuous viscosity solutions in the one-dimensional equation domain based on a new BV estimate. In our work, we aim to numerically approximate a continuous nondecreasing viscosity solution of~\eqref{non periodic equation0} without assuming a Lipschitz regularity on the kernel and on the initial data. The existence of such solution was proved by El Hajj in~\cite{1}. 

Up to our knowledge, there are only few known numerical results. Indeed, a pioneering work was done by Crandall and Lions~\cite{1350}, for approximations of solutions of a broader class of equations called Hamilton-Jacobi equations. Alvarez, Carlini, Monneau, and Rouy~\cite{250} proved the convergence of a first-order scheme for a nonlocal eikonal equation of a Lipschitz continuous solution and extended this convergence result to the numerical analysis of a nonlocal Hamilton-Jacobi equation describing the dynamics of a single dislocation in 2D. We also mention the work of A. Ghorbel, R. Monneau~\cite{50000}, where they proposed an upwind scheme to approximate a Lipschitz continuous viscosity increasing solution having a Lipschitz continuous kernel.

\subsection{Description of our results}\label{sec.description}

This paper has to main objectives:
\begin{itemize}
    \item Deriving a reliable periodic model from~\eqref{non periodic equation0}.
    \item Designing a convergent numerical scheme for such periodic model.
\end{itemize}
The first objective will allow us to get rid of the infinite size of the domain on which~\eqref{non periodic equation0} is defined. Of course, in order to ensure some ``consistency'' between the initial model and the periodic one, we need to justify that when the artificial period goes to infinity one recover the solutions of the initial problem~\eqref{non periodic equation0}. For this purpose, the key idea is to preserve, for the periodic model, a gradient entropy inequality available for~\eqref{non periodic equation0} (see below). This will be done in Section~\ref{sec.derivperiodic} by carefully modifying the kernel $\K$. 

Then, for the second main challenge, we design, in Section~\ref{sec.numsch}, a numerical scheme for the periodic equation. More precisely, we introduce a one step implicit/explicit (IMEX) Euler in time and upwind in space scheme. We will show, inspired by \cite{2}, that this scheme preserves at the discrete level the gradient entropy inequality of the periodic model. This inequality will allow us to show the convergence of the scheme.

\subsection{Gradient entropy estimate for the initial model}

As explained previously, our approach relies mainly on the gradient entropy inequality satisfy by the solutions of~\eqref{non periodic equation}. Therefore, for the convenience of the reader we give some details on this inequality and we refer, for instance, to~\cite{1} for more details.

Thereby, we consider $v$, a smooth as needed, solution of~\eqref{non periodic equation0}. Besides, we also assume that $v_0$ is smooth and nondecreasing. Then, with these assumptions one can prove 
by the classical maximum principle, that $v$ is also nondecreasing along time. This leads us to drop the absolute value in~\eqref{non periodic equation0} and then consider the following non-local transport equation: 
\begin{align}\label{non periodic equation}
    \begin{cases}
  \partial_t v(x,t) = \left( \K(\cdot) \star v(\cdot,t)(x)\right)
 \partial_xv(x,t) & \mbox{in } \mathbb{R} \times (0, T), \\
  v(x,0) = v_0(x) & \mbox{in } \mathbb{R}. 
\end{cases}
\end{align}
Now, similar to~\cite{1}, we introduce the following entropy functional:
\begin{align}\label{1.defent}
    H[\partial_x v](t) = \int_\R g\left(\pa_x v(x,t)\right) \, \dd x,
\end{align}
where $g$ is defined on $\R_+$ with
\begin{equation}\label{1.defg}
     g(x) = x \, \ln(x).        
\end{equation}
Then, one can prove that there exists a constant $C>0$ depending on $T$, $\K$ and $v_0$ such that for all $t \in (0,T)$ it holds
\begin{align}\label{1.entv}
    H[\pa_x v](t) - \int_0^t \int_\R \left(\K(\cdot)\star \pa_x v(\cdot,s)(x)\right) \, \pa_x v(x,s) \, \dd x \dd s \leq C.
\end{align}
Applying Plancherel's identity we notice that the second term in the left hand side is nonnegative if we assume that $\F(\K)(\xi) \leq 0$ for all $\xi \in \R$ where $\F$ denotes the Fourier transform. In fact, following for instance~\cite{1}, all the previous computations can be made rigorous under the following assumptions:\smallskip
\begin{itemize}
\item[\textbf{(H1)}] {\bf{Initial data}}: $v_0 \in L^\infty(\R)$ is a nondecreasing function with $\partial_x v_0 \in L^1(\mathbb{R}) \cap L\,\log\,L(\mathbb{R})$.

\item[\textbf{(H2)}] {\bf Kernel}: $\K \in L^1(\R)$ is an even function such that
\begin{equation}\label{fourier transform is negative}
\displaystyle \int_{\R} \mathcal{K}(x) \dd x = 0\quad \mbox{and}\quad 
\mathcal{F}(\mathcal{K})(\xi) \leq 0, \quad \forall \xi \in \R^{*}. 
\end{equation}
\end{itemize}
In particular, assumption~\textbf{(H1)} implies that $v_0$ is continuous. Let us also recall the definition of the so-called Zygmund space:
\begin{equation}\label{defLlogL}
    L\, \log \,L(\R) = \left\{ f \in L^1_{{\rm loc}}(\R) \, : \, \int_{\R} |f|\ln( e + |f|) \, \dd x <  \infty \right\},
\end{equation}
which is a Banach space with the norm (see~\cite{111})
\begin{align*}
\|f\|_{L\,\log \,L(\R)} = \inf \left\{ \mu > 0 \, : \, \int_{\mathbb{R}} \frac{|f|}{\mu} \,\ln \left( e + \frac{|f|}{\mu} \right) \, \dd x \leq 1 \right\}.
\end{align*}

Assumption~\textbf{(H2)} is a property that arises naturally from the physical laws governing dislocations and was proven in \cite[Proposition 6.1]{311}. In particular, the negativity of the Fourier transform of the kernel follows directly from the positivity of the elastic energy. In Section 6, we present numerical simulations based on a specific kernel derived from the Peierls–Nabarro model, developed within the framework of isotropic elasticity (see  \cite{19} and  \cite[Sec. 6.2.2]{311} for the explicit computation).

\subsection{Outline of the paper}
This paper is organized as follows. In Section~\ref{sec.periodicmodel}, we derive our periodic model and we give a formal insight on its mathematical structure. Moreover, we also justify that when the period goes to infinity one recover the initial model~\eqref{non periodic equation}. Then, in Section~\ref{sec.nummain}, we introduce our IMEX numerical scheme and state our main results. We study in Section~\ref{sec.studyscheme} its main properties on fixed grids in space and time, and we prove its convergence in Section~\ref{convergence}. Finally, we present some numerical experiments in Section~\ref{sec.numexp}.

\section{Introduction of a periodic model}\label{sec.periodicmodel}

\subsection{Derivation of a periodic model}\label{sec.derivperiodic}

We first consider for a fixed $P>0$ the following ansatz: 
\[
v^P(x, t) := v(x, t) - L^Px, \quad \mbox{for }(x,t) \in [-P,P] \times (0,T),
\]
where 
\begin{equation}\label{def.LP}
L^P=\frac{v_0(P)-v_0(-P)}{2P}, 
\end{equation} 
and $v^P$ is a $2P$-periodic function. This approach allows to describe the evolution of dislocations in a periodic domain (i.e. without boundray effect) assuming that the total dislocation density is preserved over time, which is reflected in the fact that
\begin{equation}\label{densitetotale}
\int_{-P}^P  \partial_x v(x,t) \dd x= \int_{-P}^P  (\partial_x v^P(x,t) +L^P) \dd x= 
 v_0(P)-v_0(-P),
\end{equation}
while respecting the positivity of the density ($\partial_x v \ge 0$). 
This is a way of studying the bulk behavior of the material away from its boundary  by modeling a periodic distribution of dislocations of length $2P$. Hence, we have  
 \begin{align*}
     \partial_t v^P(x,t) &= \left( \mathcal{K}(\cdot) \star v(\cdot,t)(x)\right) \, \partial_x v(x,t)= \left( \mathcal{K}(\cdot) \star v(\cdot,t)(x)\right) \,\left( \partial_x v^P(x, t) + L^P\right).
 \end{align*} 
According to~\textbf{(H2)} and further assuming the following hypothesis on the kernel: 
\begin{itemize}
 \item[\textbf{(H2)'}] Let $x \in \R \mapsto x \, \mathcal{K}(x)$ be a function in $L^1(\R)$.
 \end{itemize}
we  can see that 
 \begin{align*}
     \left( \mathcal{K}(\cdot) \star v(\cdot,t)(x)\right)
     &= \int_{\mathbb{R}}\mathcal{K}(y) v^P(x - y, t)\, \dd y+ \int_{\mathbb{R}}\mathcal{K}(y) L^P(x - y) \, \dd y\\
     &=    \left( \mathcal{K}(\cdot) \star v^P(\cdot,t)(x)\right) + L^Px \int_{\mathbb{R}}\mathcal{K}(y)\, \dd y -  L^P\int_{\mathbb{R}} y\mathcal{K}(y)\, \dd y =   \left(\mathcal{K}(\cdot) \star v^P(\cdot,t)(x)\right).
 \end{align*}
  Consequently, we deduce that  $v^P$ is solution to the following system:
 \begin{align}\label{equation involving u}
 \begin{cases}
        \partial_t v^P(x,t) = \left(\mathcal{K}(\cdot) \star v^P(\cdot,t)(x)\right) \, \left(\partial_x v^P(x, t) + L^P\right) & \mbox{in } \mathbb{R} \times (0, T),\\
        v^P(x,0) = v_0(x) - L^Px & \mbox{in } \mathbb{R}.
 \end{cases}
 \end{align}
Let us notice that systems~\eqref{non periodic equation} and~\eqref{equation involving u} are equivalent. In particular, $\pa_x v^P+L^P$ satisfies the gradient entropy estimate~\eqref{1.entv}.

\begin{remark}{$\;$}

\begin{itemize}
\item Noting that it is entirely possible to establish the results developed in this paper without assuming that the kernel verifies hypothesis {\textbf{(H2)'}}. This is a simplifying hypothesis that makes the analysis and reading easier.
\item  It can be observed from equation \eqref{equation involving u} 
 that the quantity $\int_{[-P,P]}  \partial_x v(x,t) \dd x$  is well preserved, 
which confirms  \eqref{densitetotale}.
\end{itemize}
\end{remark}
Now, our main objective is to derive a periodic model from~\eqref{equation involving u}. As already explained, we need to derive this model in such way that a ``periodic'' counterpart of~\eqref{1.entv} still holds. For this, we will carefully modify the kernel $\K$. Indeed, we first define $\K^P$ the $2P$ periodic function such that 
\[
\K^P(x) = \K(x) \quad \mbox{for }x \in [-P,P],
\]
and extended periodically over $\R$. Now, for a given positive integer $M>1$, 
we introduce  the Fej\'er kernel $F_M$ as: 
 \begin{equation}\label{Fejer kernel}
      F_M(x)= \frac{1}{M} \sum_{\ell = 0}^{M - 1}  D_{\ell}(x),  
  \end{equation}
 where $D_\ell$ is the Dirichlet  kernel  defined by: 
   $$
 D_\ell(x)=  \sum_{|m|\leq \ell} e^{i \,\pi \, m \,x/P}
      \quad\mbox{for} \quad    \ell = 0, \ldots, M-1.
 $$
Then, we define the $2P$ periodic function $\K^P_M$ by
\begin{align}\label{1.defKP}
    \K^P_M(x) = \K^P(x) - \frac{2F_{2M}(x)}{P} \, \int_{|x|\geq P} \left|\K(x)\right| \, \dd x \quad \mbox{for }x \in \R. 
\end{align}
Finally, we consider the  Ces\`aro mean of order $M$ 
 of the Fourier series of $\K^P_M$:
 \begin{align}\label{cesaro}
     \sigma_M^P(x) := \sigma_M(\K^P_{M})(x) = \dfrac{1}{M} \sum_{\ell=0}^{M-1} \sum_{|m|\leq \ell} c_m(\K^P_{M}) \, e^{i \,\pi \,m \, x/P} = \frac{1}{2P} \left(F_M \ast \K^P_{M}\right)(x), \quad \forall x \in I_P,
 \end{align}
 where, throughout the paper, we will denote by $I_P$ the set
 \begin{align}\label{defI}
     I_P := \R/2P\Z,
 \end{align} 
 and by $c_m(\K^P_M)$  the $m$-th Fourier coefficient of $\K^P_M$ given by
 \begin{align*}
     c_m(\K^P_M) = \frac{1}{2P} \int_{I_P} \K^P_M(x) \, e^{-i\pi m \, x/P} \, \dd x, \quad \forall m \in \Z.
 \end{align*}
Moreover, in the definition~\eqref{cesaro} the symbol $\ast$ denotes the convolution between $2P$ periodic functions defined, for any $w$ and $z\in L^1(I_P)$, as
\begin{align}\label{conv.prod.perio}
    (w\ast z)(x) = \int_{I_P} w(x-y) \, z(y) \, \dd y, \quad \mbox{for a.e. }x \in I_P.
\end{align}
Hence, in this paper, we will study the solutions, denoted $u^P_M$, of the following periodic system:
 \begin{align}\label{1.equPaux}
 \begin{cases}
\partial_t u^P_M(x, t) = \left(\sigma_M^P(\cdot) \ast u^P_M(\cdot,t)(x)\right) \, \left(\partial_x u^P_M(x, t) + L^P \right) & \mbox{in } I_P \times (0, T),\\
u^P(x,0) = u^P_0(x) & \mbox{in } I_P,
\end{cases}
 \end{align}
 where $I_P$ is defined in~\eqref{defI}. Let us assume that $P$, $M$ and $u^P_0$ satisfy the following assumptions:

 \begin{itemize}
\item[\textbf{(H3)}] {\bf Parameters of the periodic model:} $P$ is a real positive number and $M$ is a positive integer.

\item[\textbf{(H4)}] {\bf Periodic initial data:} for any $P >0$, the function $u^P_0$ belongs to $L^\infty(I_P)$ with $\pa_x u^P_0 \in L^1(I_P) \cap L\,\log\,L(I_P)$. Moreover, the function 
\begin{align*}
    x \in [-P,P) \mapsto u^P_0(x) + L^Px,
\end{align*}
is nondecreading.

\end{itemize}

\begin{remark}
    The space $L\,\log\,L(I_P)$ is defined as in~\eqref{defLlogL} with $\R$ replaced by $I_P$.
\end{remark}
We show in this paper that under assumptions~\textbf{(H2)}--\textbf{(H4)},  system~\eqref{1.equPaux} admits at least one solution $u^P_M$ in the distributional sense (see  Theorem~\ref{convergence in periodic}). In the next section we explicit the gradient entropy estimate satisfy by the solutions of~\eqref{1.equPaux}. The adaptation at the discrete level of such estimate will be crucial to show the existence of solutions for this system, see the proof of Theorem~\ref{convergence in periodic}.

 \subsection{Gradient entropy estimate in the periodic case}\label{subsec.gredentperio}

Let us now establish formally a gradient entropy inequality, similar to~\eqref{1.entv}, satisfies by $\pa_x u^P_M+L^P$. This will (partially) motivate the introduction of 
 the Ces\`aro means of the Fourier series of the $2P$-periodic kernel  $\K^P_M$. 
 Moreover, as already said, this inequality will be at the cornerstone  of the proof of our convergence result announced in Theorem~\ref{convergence in periodic}. Indeed, as in~\cite{1}, this gradient entropy inequality will lead to a crucial estimate on $\pa_x u^P_M + L^P$ in the space $L \log L(I_P)$ (see Lemma~\ref{Llog L estimate}).

Therefore, as previously, in this part we will assume that $u^P_M$ is smooth as needed and we will also assume that the function $x \mapsto u^P_M(x,t) + L^P\, x$ is increasing. Then, recalling definition~\eqref{1.defent} of the functional $H$, with $\R$ replaced by $I_P$, we have
\begin{align*}
    \frac{\dd}{\dd t}H\left[\pa_x u^P_M + L^P\right](t) & = 
    \int_{I_P} g'(\pa_x u^P_M(x,t)+L^P)\, \pa_t \pa_x u^P_M(x,t) \, \dd x\\
    &= -\int_{I_P} \pa_x^2 u^P_M(x,t)\, g''(\pa_x u^P_M(x,t) + L^P) \pa_t u^P_M(x,t) \, \dd x\\
    &= -\int_{I_P} \pa_x^2 u^P_M(x,t) \, \left(\sigma_M^P(\cdot) \ast u^P_M(\cdot,t)(x) \right) \, \dd x\\
    &= \int_{I_P} \pa_x u^P_M(x,t) \, \left(\sigma_M^P(\cdot) \ast \pa_x u^P_M(\cdot,t)(x)  \right) \, \dd x.
\end{align*}
Now, thanks to Parseval's relation and the property of the Fourier coefficients with respect to the convolution product, we obtain
\begin{align*}
    \frac{\dd}{\dd t} H\left[\pa_x u^P_M + L^P\right](t)= (2P)^2 \sum_{m \in \Z} c_m(\sigma_M^P) \, \left|c_m(\pa_x u^P_M)\right|^2.
\end{align*}
Therefore, to argue as in the case of system~\eqref{non periodic equation}, we need to show that the Fourier coefficients of $\sigma_M^P$ are nonpositive real numbers. First, thanks to the definition~\eqref{cesaro} of $\sigma_M^P$ and the definition~\eqref{1.defKP} of $\K^P_M$, we notice that
\begin{align*}
    c_m(\sigma_M^P) = c_m(F_M) \, c_m(\K^P_M) = c_m(F_M) \, \left(c_m(\K^P) - \frac{2c_m(F_{2M})}{P} \, \int_{|x|\geq P} |\K(x)| \, \dd x \right), \quad \forall m \in \Z.
\end{align*}
Besides, classical computations lead to
\begin{align*}
    c_m(F_M) =
    \left\{
    \begin{array}{ll}
        \left(1-\frac{|m|}{M}\right) & \mbox{if } |m| < M, \\
        0 & \mbox{if } |m|\geq M.
    \end{array}
\right.
\end{align*}
Therefore, we have
\begin{align}\label{auxcoeffFouriercesaro}
    c_m(\sigma_M^P) =
        \left\{
    \begin{array}{ll}
       \left(1-\frac{|m|}{M}\right)\left(c_m(\K^P) - 
       \frac{2}{P}\left(1-\frac{|m|}{2M}\right) \displaystyle\int_{|x|\geq P} |\K(x)| \, \dd x\right) & \mbox{if } |m|< M, \\
        0 & \mbox{if } |m|\geq M.
    \end{array}
    \right.
\end{align}
Let us establish now the following lemma. 

\begin{lemma}\label{fourier series of kpdelta negative} 
Let assumption~\textbf{\emph{(H2)}} holds. Then, the Fourier coefficients of the $2P$-periodic function $\sigma_M^P$ are nonpositive real numbers.
\end{lemma}

 \begin{proof}
 According to~\eqref{auxcoeffFouriercesaro} it is sufficient to show that $c_m(\K^P_M)$ is a nonpositive real number for $|m|< M$. Then, let $|m|< M$ be fixed. We first notice, thanks to assumption~\textbf{(H2)}, that the function $\K^P_M$ is even, so that its Fourier coefficient are real numbers. Now, we subtract $\int_{-P} ^{P}\K^P(x) e^{-i \,\pi \,m \,x/P}$ from the Fourier transform of $\mathcal{K}$ evaluated at $\xi = m/2P$, we have
 \begin{align*} 
       \bigg| \F(\K)(m/2P) -  \int_{-P} ^{P}\K^P(x) \,e^{-i\, \pi\, m\, x/P} \, \dd x \bigg|  &\leq  \bigg| \int_{|x| \geq P}\mathcal{K}(x) e^{-i\, \pi\,m\, x/P} \dd x \bigg|\leq \int_{|x| \geq P} |\mathcal{K}(x)| \, \dd x. 
 \end{align*}
Moreover, using~\textbf{(H2)}, we obtain
 \begin{align*}
     \int_{-P}^{P}\K^P(x) e^{-i\, \pi\, m \,x/P} \, \dd x &\leq  \int_{|x| \geq P} |\mathcal{K}(x)| \, \dd x + \int_{\mathbb{R}}\mathcal{K}(x) \, e^{-i \,\pi\, m \, x/P }\, \dd x \leq \int_{|x| \geq P} |\mathcal{K}(x)| \, \dd x. 
 \end{align*}
Hence,
 \begin{align}\label{2.aux}
      \int_{-P}^{P}\K^P(x) e^{- i\, \pi\,m \,x/P} \, \dd x - 2 \int_{|x| \geq P}|\mathcal{K}(x)| \, \dd x \leq -  \int_{|x| \geq P} |\mathcal{K}(x)| \, \dd x  \leq 0.
 \end{align}
Therefore, as $|m|<M$, we get
\begin{align*}
    c_m(\K^P_M) &= c_m(\K^P) - \frac{2}{P} \left(1-\frac{|m|}{2M}\right) \, \int_{|x|\geq P} |\K(x)| \, \dd x \\
    & \le \frac{1}{2P} \int_{-P}^P \K^P(x) \, e^{-i \, \pi \, m \, x/P} \, \dd x - \frac{1}{P} \int_{|x|\geq P} |\K(x)| \, \dd x\\
    & \le\frac{1}{2P} \left(\int_{-P}^P \K^P(x) \, e^{-i \, \pi \, m \, x/P} \, \dd x - 2\int_{|x|\geq P} |\K(x)| \, \dd x \right) \leq 0,
\end{align*}
where for the last inequality we apply~\eqref{2.aux}.
 \end{proof}

The previous computations allow to understand that in order to preserve the gradient entropy inequality for the periodic model one needs to consider the Ces\`aro means of $\K^P_M$. Indeed, in the definition~\eqref{1.equPaux}, we cannot simply consider $\K^P_M$ since
\begin{align*}
    c_m(\K^P_M) =
        \left\{
    \begin{array}{ll}
        c_m(\K^P) - \frac{2}{P}\left(1-\frac{|m|}{2M}\right)\left(\displaystyle\int_{|x|\geq P} |\K(x)|\,\dd x\right) & \mbox{if } |m|< 2M, \\
        c_m(\K^P) & \mbox{if } |m|\geq 2M,
    \end{array}
    \right.
\end{align*}
and $c_m(\K^P)$ is a priori an unsigned quantity with our assumption~\textbf{(H2)}. Moreover, in order to be able to pass to the limit $P\to +\infty$ and to design our numerical scheme, we also need to regularize the periodic kernel $\K^P_M$. This is done thanks to its Ces\`aro means.

\subsection{From the periodic to the initial model}\label{sec.fromperiodic.to.init}

In this section, we explain how one can recover the initial model~\eqref{non periodic equation} from~\eqref{1.equPaux} when the period $P \to +\infty$. For this purpose, for a given function $v_0$ satisfying~\textbf{(H1)}, we first consider the following $2P$-periodic initial data for~\eqref{1.equPaux}:
\begin{align}\label{1.initperiod}
    u^P_0(x) = v_0(x) - L^P \, x, \quad \mbox{for } x \in [-P,P],
\end{align}
which we extend periodically over $\R$, where $L^P$ is defined in  \eqref{def.LP}. 
We notice that for any $P>0$, the quantity $L^P$ is nonnegative if $v_0$ is nondecreasing. Then, we have the following

\begin{lemma}\label{lem.init.perio.to.nonperiod}
Let assumption~\textbf{\emph{(H1)}} holds and assume that $P\ge 1$. Then, the function $u^P_0$, given by~\eqref{1.initperiod}, satisfies, uniformly with respect to $P$,
\begin{align*}
    u^P_0 \in L^\infty(I_P), \quad \pa_x u^P_0 \in L^1(I_P)\cap L\,\log \,L(I_P).
\end{align*}    
Moreover, the continuous function $x \in [-P,P)  \mapsto u^P_0(x) + L^Px$ is nondecreasing.
\end{lemma}

\begin{proof}
Let $x \in [-P,P)$, then, since $P>0$, we have
\begin{align}\label{estimationLinP}
    \|u^P_0\|_{L^\infty(I_P)}\leq \|v_0\|_{L^\infty(\R)} + \frac12 \left|v_0(P)-v_0(-P)\right|\leq 2\|v_0\|_{L^\infty(\R)},
\end{align}
so that $u^P_0 \in L^\infty(I_P)$ uniformly w.r.t $P$. Moreover, as $P>0$, it holds
\begin{align}\label{estimationL1P}
    \|\pa_x u^P_0\|_{L^1(I_P)} \leq \|\pa_x v_0\|_{L^1(\R)} + 2\|v_0\|_{L^\infty(\R)} \le  4\|v_0\|_{L^\infty(\R)}. 
\end{align}
Then, applying inequality~\eqref{w in llogl} in Lemma~\ref{Llog L estimate} to the constant function equal to one, we obtain
\begin{align}\label{estimationLlogLP}
    \|\pa_x u^P_0\|_{L\,\log\,L(I_P)} 
    &\leq  \|\pa_x v_0\|_{L\,\log\,L(I_P)} + \|v_0\|_{L^\infty(\R)} 
    \left(\frac{1}{P}+ 2\ln(1+e^2)+\frac{2}{e}\right), \\
    & \leq \|\pa_x v_0\|_{L\,\log\,L(\R)} + \|v_0\|_{L^\infty(\R)} 
    \left(1+ 2\ln(1+e^2)+\frac{2}{e}\right),
\end{align}
where we have used the assumption $P\ge 1$ for the last inequality. 
\end{proof}

Therefore, the initial condition~\eqref{1.initperiod} satisfies assumption~\textbf{(H4)}. In particular, assuming that assumptions $\textbf{(H1)}$-$\textbf{(H2)}$ also hold, we deduce from Theorem~\ref{convergence in periodic}, that for any $P\geq 1$ and $M>0$, there exists a solution $u^P_M$ in the distribution sense to~\eqref{1.equPaux} satisfying all the properties listed in Theorem~\ref{convergence in periodic}.\\
 
 Then, we have the following: 

\begin{proposition}\label{prop.PMinfty}
    Let assumptions~~\textbf{\emph{(H1)}}-\textbf{\emph{(H2)}} hold, let $(M_\ell)_{\ell \in \N}$ be a sequence of positive integers such that $M_\ell \to+\infty$ as $\ell\to+\infty$, and let $(P_k)_{k\in\N}$, with $P_0\geq 1$, be a nondecreasing sequence such that $P_k \to +\infty$ as $k\to+\infty$. Let also $(u_{k,\ell})_{k,\ell \in \N}$, with $u_{k,\ell} :=u^{P_k}_{M_\ell}$, be the associated sequence of solutions to~\eqref{1.equPaux} with $u^{P_k}_0$ given by~\eqref{1.initperiod}. Then, there exists $v \in C(\R \times [0,T))$ with
    \begin{align*}
        v \in L^\infty(\R \times (0,T)) \cap C(0,T;L\,\log\,L(\R)), \quad \pa_x v \in L^\infty(0,T;L\,\log\,L(\R)),
    \end{align*}
    such that the family $(u_{k,\ell})_{k,\ell \in\N}$ converges, up to a subsequence, uniformly on every compact set of $\R \times(0,T)$ to $v$ when first $\ell\to+\infty$ and then $k\to+\infty$. Moreover, $v$ is a solution of~\eqref{non periodic equation} in the distributional sense. 
\end{proposition}

The proof of Proposition~\ref{prop.PMinfty} is quite similar to the proof of Theorem~\ref{convergence in periodic}. For this reason, we postpone its proof in Appendix~\ref{app.proofPMinfty}, see also Remark~\ref{rem.constansMP}.

\begin{remark}
Note that Proposition~\ref{prop.PMinfty}  ensures some ``consistenc'' between the periodic model and the initial model, its proof relies on compactness arguments. It would also be interesting to obtain an explicit error estimate, with respect to $P$, $M$ and an appropriate distance, between $u^P_M$ and $v$ in order to better understand the behavior of the approximate solution $u^P_M$ with respect to these parameters. This is strongly linked to the uniqueness issue which we do not address in this paper and we leave its study for further research.
\end{remark}

\section{Numerical scheme and main results}\label{sec.nummain}

\subsection{Numerical scheme}\label{sec.numsch}

We are now in position to introduce a numerical scheme for the periodic system~\eqref{1.equPaux}. Let us first introduce some notation for the meshes in space and time. Let $N \ge M$ and $N_T$ be positive integers. We define a time step $\Delta t > 0$ such that $\Delta t = T/N_T$ and a space step $\Delta x = P/N$. We also define the sequences
\[
t_n =n\Delta t, \quad \forall n \in \{0,\ldots, N_T\},
\]
 and
\[
x_i = i \Delta x, \quad \forall i\in \I := \Z/2N\Z.
\]

For a fixed positive integer $M$,  we first start by discretizing
the $2P$ periodic initial condition $u^P_0$ as
\begin{align}\label{schemeIC}
u_i^{P, 0} =  u^P_0(x_i), \quad \forall i \in \I.
\end{align}
Then, for a given vector $u^{P,n} = (u^{P, n} _i)_{i \in \I}$, where $u^{P,n}_i$ is an approximation of $u^P(x_i,t_n)$,  we define the vector $u^{P,n+1}= (u^{P, n+1} _i)_{i \in \I}$ as the solution of the following system: 
\begin{align*}
\begin{cases}
  \dfrac{u_i^{ P, n+1} - u_i^{P, n}}{\Delta t} &=   \lambda^M_i \left[u^{P, n + 1}\right] \, \left(\theta^{P,n}_{i+1/2} + L^P \right)  \mbox{ if }   \lambda^M_i \left[u^{P, n + 1}\right] \geq 0, \quad \forall i \in \I,\\ \\
  \dfrac{u_i^{P, n+1} - u_i^{P, n}}{\Delta t} &=  \lambda^M_i \left[u^{P, n + 1}\right] \, \left(\theta^{P,n}_{i-1/2}+ L^P \right)  \mbox{ if }   \lambda^M_i \left[u^{P, n + 1}\right] < 0, \quad \forall i \in \I, 
\end{cases}
\end{align*}
where $ \theta^{P,n}_{i+1/2}$ is the discrete gradient of $u^{P, n} _i$  defined by 
\begin{align}\label{def.dis.grad}
    \theta^{P,n}_{i+1/2} := \dfrac{u^{P,n}_{i+1}-u^{P,n}_i}{\Delta x}, \quad \forall i\in\I.
\end{align}
Before to define the approximate velocity term $\lambda^M_i\left[u^{P,n+1}\right]$, we notice that we can rewrite the previous system as
\begin{equation}\label{numerical compacted scheme}
    \frac{u_i^{P, n+1} - u_i^{P, n}}{\Delta t} =  \lambda^M_i\left[u^{P, n + 1}\right]_+ \, \theta^{P,n}_{i+1/2} - \lambda^M_i\left[u^{P, n + 1}\right]_- \, \theta^{P,n}_{i-1/2}  + L^P \, \lambda^M_i\left[u^{P, n + 1}\right], \quad \forall i \in \I,
\end{equation}
where
\[
\lambda^M_i [\cdot]_+ = \frac{1}{2} \,\Big(\lambda^M_i [\cdot] + \left|\lambda^M_i [\cdot]\right|\Big), \quad \lambda^M_i [\cdot]_- = \frac{1}{2} \,\Big(\left|\lambda^M_i [\cdot]\right| - \lambda^M_i [\cdot] \Big), \quad \forall i \in \I.
\]
Finally, we define
  \begin{equation}\label{h(x)}
  \lambda^M_i \left[u^{P, n + 1}\right]= \sum_{j \in \I} \Delta x \, \sigma_{M,j}^P \, u^{P, n + 1}_{i - j}, \quad \forall i \in \I, 
 \end{equation}
 with $\sigma_{M,j}^P=\sigma_{M}^P(x_j)$ and where $\sigma_{M}^P$ is the Ces\`aro mean of order $M$ 
 of the Fourier series of $\K^P_M$, introduced previously in \eqref{cesaro}. 
 
\begin{remark}
 It is clear that the vector $u^{P,n}$, for any $1\leq n\leq N_T$, solution to~\eqref{numerical compacted scheme}, depends on $M$. However, when no confusion can occur, we make a slight abuse of notation and we neglect $M$ from its indices.
\end{remark}

\begin{remark}
There exists an alternative definition of the term $\sigma^P_{M,j}$. Indeed, we can alternatively define this term as
\begin{align}\label{def.alternative}
\sigma^P_{M,j} = \frac{1}{2P} \sum_{i \in \I} \Delta x \, F_M(x_{j-i}) \, \K^P_{M,i},
\end{align}
where, defining $x_{i\pm 1/2} := (i\pm 1/2) \Delta x$ for any $i \in \I$, we set
\begin{align*}
\K^P_{M,i} = \frac{1}{\Delta x} \int_{x_{i-1/2}}^{x_{i+1/2}} \K^P_M(x) \, \dd x, \quad \forall i \in \I.
\end{align*}
The same results presented in this paper can also be obtained using definition~\eqref{def.alternative}  (see Remarks ~\ref{rem.alter1} and \ref{rem.alter2}). However, as will be shown below, adopting definition~\eqref{cesaro} yields uniform estimates with respect to all the parameters involved ($N$, $N_T$, $M$, and $P$), in contrast to definition ~\eqref{def.alternative}, where the estimates are uniform only with respect to $N$ and $N_T$ (cf. estimate \eqref{est_N_NT}). While this distinction does not hinder the proof of convergence of the scheme, the formulation based on definition~\eqref{cesaro} offers greater robustness and stability under parameter variation.
\end{remark}

\subsection{Main results}
Our first main result deals with the well-posedness of the scheme~\eqref{schemeIC}--\eqref{h(x)} and the qualitative properties of its solution. Before to state properly this result, we slightly modify the definition of the entropy density $g$, defined in~\eqref{1.defg}, by the following nonnegative function
\begin{equation}\label{1.deff}
    0 \leq f(x) = \begin{cases}
        x \ln(x) + \frac{1}{e} & \quad\mbox{if } x\geq 1/e,\\
        0  & \quad \mbox{if } 0 \leq x \leq  1/e.
    \end{cases}          
\end{equation}

\begin{theorem}[Properties of the scheme]\label{thm.scheme}
Let assumptions~\textbf{\emph{(H2)}}--\textbf{\emph{(H4)}} hold, and assume that
\begin{equation}\label{delta t}
    \Delta t < \frac{1}{10 (L^P) \|\mathcal{K}\|_{L^1 (\mathbb{R})}}\,\min\left\{\frac{1}{\|u^P_0\|_{L^\infty(I_P)}\, e^{10 (L^P) T ||\mathcal{K}||_{L^1(\mathbb{R})}} + 1}, 1\right\},
\end{equation}
and
\begin{equation}\label{delta t over delta x}
   \frac{\Delta t}{\Delta x} < \frac{1}{10 \|\mathcal{K}\|_{L^1 (\mathbb{R})}}
    \min\left\{ 
\frac{1}{\left(\|u^P_0\|_{L^\infty(I_P)}\, e^{10 (L^P)T \|\mathcal{K}\|_{L^1(\mathbb{R})}} + 1\right)}, 
\frac{1}{3 \|u^P_0\|_{L^\infty(I_P)}\, e^{10 (L^P)T \|\mathcal{K}\|_{L^1(\mathbb{R})}}} 
    \right\}.
\end{equation}
Then, for any $1\leq n\leq N_T$, the scheme~\eqref{schemeIC}--\eqref{h(x)} admits a unique solution $u^{P,n} = (u^{P,n}_i)_{i\in\I}$ such that
\begin{align}\label{Linftybounddis}
 \max_{ i \in \I} \, |u_{i}^{P, n}| & \leq \|u^P_0\|_{L^\infty(I_P)} \, e^{10 (L^P) T \|\mathcal{K}\|_{L^1(\mathbb{R})}}, 
\end{align} 
and $\theta_{i + 1/2}^{P, n} + L^P$ is nonnegative for any $i\in\I$ where we recall definition~\eqref{def.dis.grad} of $\theta^{P,n}_{i+1/2}$. Furthermore, this solution satisfies the following discrete total variation estimate:
\begin{align}\label{TVestim}
 TV(u^{P, n} )= \sum_{i \in \I} \left| u^{P,n}_{i+1}-u^{P,n}_i\right| \leq TV(u^P_0 )\, e^{10 (L^P) T \|\mathcal{K}\|_{L^1(\mathbb{R})}}, \quad \forall 1\leq n \leq N_T.
 \end{align}
Finally, there exists a positive constant $\zeta_0$ only depending on $T$ and $\|\mathcal{K}\|_{L^1(\mathbb{R})}$ (see~\eqref{defzeta}), such that the following discrete gradient entropy estimate holds
 \begin{align}\label{gradentdis}
 \sum_{i \in\I} \Delta x \, f\left(\theta_{i + 1/2}^{P, n} + L^P\right) \leq \sum_{i \in \I} \Delta x \, f\left(\theta_{i + 1/2}^{P, 0}+ L^P\right) + \zeta_0 \, TV(u^P_0), \quad \mbox{for } 1\leq n \leq N_T,
\end{align}
where we recall definition~\eqref{1.deff} of the function $f$.
\end{theorem}

The proof of Theorem~\ref{thm.scheme} is done in Section~\ref{sec.studyscheme} in several steps. First we prove the existence of a unique bounded solution, satisfying~\eqref{TVestim}, to the scheme~\eqref{schemeIC}--\eqref{h(x)} thanks to Banach's fixed point theorem. Then, we adapt at the discrete level the computations of Section~\ref{subsec.gredentperio} to show the discrete gradient entropy estimate~\eqref{gradentdis}.

\begin{remark}
Let us notice that all the properties stated in Theorem~\ref{thm.scheme} hold under the ``uniform'' CFL like conditions~\eqref{delta t}--\eqref{delta t over delta x}. However, thanks to an adaptive time stepping strategy, these conditions on the sizes of the meshes can be relaxed. More precisely, we first define $t_0 = 0$ and
\[
0< \Delta t_0 < \dfrac{1}{10 (L^P) \|\K\|_{L^1(\R)}}.
\]
Then, we introduce a sequence of points $(t_{n+1})_{n\geq 0}$ with $t_{n+1} = t_n + \Delta t_{n+1}$ where $\Delta t_{n+1}$ satisfies the conditions
\begin{align*}
\Delta t_{n+1} < \min\left(\Delta t_n, \dfrac{1}{10 (L^P) \|\K\|_{L^1(\R)} \left(\|u^{P,n}\|_\infty+1\right)}\right),
\end{align*}
and
\begin{align*}
\dfrac{\Delta t_{n+1}}{\Delta x} < \min\left(\dfrac{\Delta t_n}{\Delta x}, \dfrac{1}{10 \|\K\|_{L^1(\R)} \left(\|u^{P,n}\|_\infty+1\right)}\right),
\end{align*}
with $\|u^{P,n}\|_\infty = \max_{i \in \I} |u^{P,n}_i|$. Under these conditions and adapting the proof of Theorem~\ref{thm.scheme}, one can prove the well-posedness of the scheme~\eqref{schemeIC}--\eqref{h(x)} (with $\Delta t$ replaced by $\Delta t_{n+1}$), the nonnegativity of $\theta^{P,n}_{i+1/2}+L^P$ for any $i \in \I$ and the following (non uniform) $L^\infty$ estimate $\|u^{P,n+1}\|_\infty \leq 2 \|u^{P,n}\|_\infty$.
\end{remark}

In our second main result we prove the convergence of the scheme~\eqref{schemeIC}--\eqref{h(x)}. In order to precisely state this result, we introduce some notation. Let the assumptions of Theorem~\ref{thm.scheme} hold. Then, for any $\Delta x$ and $\Delta t$ (satisfying~\eqref{delta t}--\eqref{delta t over delta x}), we introduce the quantity $\eps := (\Delta x,\Delta t)$, and the following function:
\begin{align*}
    u^{P,\eps}_M(x_i,t_n) = u^{P,n}_i, \quad \forall i \in \I, \, 0\leq n \leq N_T,
\end{align*}
where of course $u^{P,n}$ is the unique solution of the scheme~\eqref{schemeIC}--\eqref{h(x)} obtained in Theorem~\ref{thm.scheme}. We also define, for any $(x, t) \in [x_i, x_{i + 1}] \times [t_n, t_{n+1}]$, its so-called $Q^1$ extension, still denoted $u^{P,\eps}_M$, as 
\begin{multline}\label{defQ1ext}
    u^{P, \eps}_M (x,t) := \bigg(\frac{t - t_n}{\Delta t}\bigg) \bigg\{\bigg(\frac{x - x_i}{\Delta x}\bigg) u_{i + 1} ^{P, n + 1} + \bigg(1 - \frac{x - x_i}{\Delta x}\bigg) u_i^{P, n + 1} \bigg\} \\ +\bigg(1 - \frac{t - t_n}{\Delta t}\bigg)\bigg\{\bigg(\frac{x - x_i}{\Delta x}\bigg) u_{i + 1}^{P, n} + \bigg(1 - \frac{x - x_i}{\Delta x}\bigg) u_i^{P, n }\bigg\},
\end{multline}
so that $u^{P,\eps}_M(x_i,t_n)=u^{P,n}_i$. In the sequel, we will consider a sequence $\eps_m=(\Delta x_m, \Delta t_m)$ such that $\eps_m \to (0,0)$ as $m \to \infty$, where $\Delta x_m$ and $\Delta t_m$ satisfy~\eqref{delta t}--\eqref{delta t over delta x} for all $m \in \mathbb{N}$. Moreover, we will denote the sequence of associated reconstruction function $(u^{P,m}_M)_{m\in\N}$ defined through
\begin{align*}
    u^{P,m}_M(x,t) = u^{P, \eps_m}_M(x,t), \quad \forall (x,t) \in I_P \times (0,T).
\end{align*}
Then, we claim that the sequence $(u^{P,m}_M)_{m \in \mathbb{N}}$ converges toward a solution of~\eqref{1.equPaux}, when $m\to+\infty$, in the distribution sense. More precisely, we have the

\begin{theorem}[Convergence of the scheme]\label{convergence in periodic} 
Assume that the assumptions of Theorem~\ref{thm.scheme} hold. Let $\eps_m=(\Delta x_m, \Delta t_m)$ be a sequence such that $\eps_m \to (0,0)$ as $m \to +\infty$, where $\Delta x_m$ and $\Delta t_m$ satisfy~\eqref{delta t}--\eqref{delta t over delta x} for all $m \in \mathbb{N}$. Let $(u^{P,m}_M)_{m\in\N}$ be a family of solution to the scheme~\eqref{schemeIC}--\eqref{h(x)}. Then, there exists a function $u^P_M$ with
\begin{equation}\label{first property}
     u^P_M \in L^{\infty}(I_P\times (0, T)) \cap C(0, T; L \log L(I_P)), \quad \partial_x u^P  \in L^{\infty}(0, T;  L \log L(I_P)),
 \end{equation}
such that the family $(u^{P,m}_M)_{m\in\N}$ converges uniformly, up to a subsequence, to $u^P_M$ as $m\to +\infty$.
Moreover, $u^P_M$ is a solution of~\eqref{1.equPaux} in the distributional sense, and satisfies the following
\begin{itemize}
\item {\bf $L^{\infty}$ estimate}:
\begin{equation}\label{second property}
     \|u^P_M \|_{L^{\infty}(I_P \times(0, T))} \leq e^{10(L^P) T ||\mathcal{K}||_{L^1(\mathbb{R})}} ||u_0^P||_{L^{\infty}(I_P)}.
\end{equation}
\item {\bf Gradient entropy estimate}: 
\begin{equation}\label{third property}
   \|\partial_x u^P_M \|_{L^{\infty}(0, T; L \log L (I_P) )} + \|\partial_t u^P_M \|_{L^{\infty}(0, T; L \log L (I_P) )} \leq C_{\rm GE}, 
\end{equation}
where $C_{\rm GE}=C_{\rm GE}\left(\|v_0 \|_{L^{\infty}}, T, ||\mathcal{K}||_{L^1(\mathbb{R})}\right) > 0$.\smallskip
 \item {\bf Continuity of the solution}: the solution $u^P_M$ belongs to $C(I_P\times [0,T))$ and there exists a modulus of continuity $\omega(\gamma,h)$ such that for all $\gamma, h \geq 0$ and all $(x, t) \in I_P \times (0, T-h )$, we have:
 \begin{align}\label{mod.cont.period}
|u^P_M (x + \gamma,t+h) - u^P_M (x,t)| \leq 6 C_{\rm GE}\, \omega(\gamma,h), \quad \mbox{with } \omega(\gamma, h) =\frac{1}{\ln(1+ \frac{1}{\gamma})}+ \frac{1}{\ln(1 + \frac{1}{h})},
\end{align}
where $C_{\rm GE}$ is the constant in~\eqref{third property}.
 \end{itemize}
\end{theorem}

The proof of Theorem~\ref{convergence in periodic} is done in Section~\ref{convergence}. We first establish  $\eps$-uniform estimates on the function $u^{P,\eps}_M$ given by~\eqref{defQ1ext}. Then, these estimates yields enough compactness properties on the sequence $(u^{P,m}_M)$ in order to obtain the existence of a function $u^P_M$ satisfying~\eqref{first property}--\eqref{mod.cont.period}. Finally, in Section~\ref{sec.identi}, we identify $u^P_M$ as a solution to~\eqref{1.equPaux} in the distributional sense.

\begin{remark}\label{rem.constansMP}
The proof of Proposition~\ref{prop.PMinfty} also relies on compactness methods. Of course in this case we need to establish uniform w.r.t.~$P$ and $M$ estimates. In fact these estimates are consequences of~\eqref{third property}. In particular, we will justify in Section~\ref{sec.unifestimates} that the constant $C_{\rm GE}$ is independent of $P$ and $M$.
\end{remark}

\section{Study of the numerical scheme}\label{sec.studyscheme}

In this section we prove Theorem~\ref{thm.scheme}. For this purpose, we split our proof in two main steps. Indeed, in Section~\ref{sec.wellpos}, we prove the well-posedness of the scheme, while in Section~\ref{sec.disgradent}, we prove the discrete gradient entropy inequality~\eqref{gradentdis}. But first, we establish a technical result.

\begin{lemma}\label{kpdelta less than k}
Let assumption~\textbf{\emph{(H3)}} holds and assume that $\K \in L^1(\R)$. Then, $\K^P_M$ given by~\eqref{1.defKP} satisfies $\K^P_M \in L^1(I_P)$ with
\begin{align}\label{normKPM}
    \|\mathcal{K}^p_M\|_{L^1(I_P)} &\leq 5 \|\mathcal{K}\|_{L^1(\mathbb{R})}.
\end{align}
Moreover, it holds
  \begin{align}\label{Normcesaro}
     \sum_{j \in \I} \Delta x \, \left|\sigma_{M,j}^P\right| \leq 5 \|\mathcal{K}\|_{L^1(\mathbb{R})}.
    \end{align}
\end{lemma}

\begin{proof}
Let us first prove~\eqref{normKPM}. For this purpose, using the definition~\eqref{1.defKP} of $\K^P_M$ and the classical result $\|F_{2M}\|_{L^1(I_P)}=2P$, we have
\begin{align*}
     \|\mathcal{K}^P_M\|_{L^1(I_P)} &\leq \|\K^P\|_{L^1(I_P)} + \frac{2\|\K\|_{L^1(\R)}}{P}\int_{-P}^P |F_{2M}(x)|\,\dd x = \|\K^P\|_{L^1(I_P)} + 4\|\K\|_{L^1(\R)},
\end{align*}
which yields estimate~\eqref{normKPM}. Now, applying relation~\eqref{cesaro} we have
\begin{align*}
    \sum_{j \in \I} \Delta x \, \left|\sigma_{M,j}^P\right| \leq \frac{\Delta x}{2P}  \int_{-P}^{P} \left( \sum_{j\in\I} F_M(x_j-x) \right)\, \left|\K^P_M(x)\right| \, \dd x.
 \end{align*}
 Furthermore, one can easily show that $\sum_{j \in \I} F_M(x_j-x) = 2N$ when $N\ge M$. Hence, since $\Delta x= P/N$, we get
 \begin{align*}
    \sum_{j \in \I} \Delta x \, \left|\sigma^P_{M,j}\right| \leq \frac{N \,\Delta x}{P}  \int_{-P}^{P} \left|\K^P_M(x)\right| \, \dd x = \|\K^P_M\|_{L^1(I_P)}.
 \end{align*}
 This leads to~\eqref{Normcesaro} thanks to~\eqref{normKPM}.
\end{proof} 

\begin{remark}\label{rem.alter1}
The estimate~\eqref{Normcesaro} still holds with the definition~\eqref{def.alternative} of $\sigma^P_{M,j}$. Indeed, we have
\begin{align*}
\sum_{j \in \I} \Delta x \, \left|\sigma_{M,j}^P\right| &\leq \frac{1}{2P} \sum_{i \in \I} \Bigg(\int_{x_{i-1/2}}^{x_{i+1/2}} |\K^P_M(x)| \, \dd x \Bigg)  \Bigg(\sum_{j \in \I} \Delta x \, F_M(x_j-x_i) \Bigg)\\
&\leq \frac{N \Delta x}{P} \sum_{i \in \I} \int_{x_{i-1/2}}^{x_{i+1/2}} |\K^P_M(x)| \, \dd x = \|\K^P_M\|_{L^1(I_P)}.
\end{align*}
Therefore, except for the discrete gradient entropy estimate (see Remark~\ref{rem.alter2}), the proof of Theorem~\ref{thm.scheme} below is unchanged.
\end{remark}

\subsection{Well-posedness of the scheme}\label{sec.wellpos}

Let us now prove the well-posedness of the scheme thanks to Banach's fixed point theorem. For this, we argue by induction. Thus, let $0\leq n \leq N_T-1$ be fixed, and let $u^{P,n} = (u^{P,n}_i)_{i\in\I}$ be a given vector such that 
\[
\max_{i\in\I}\, |u^{P,n}_i| \leq \alpha := \|u^P_0\|_{L^\infty(I_P)} \, e^{10 (L^P)T \|\mathcal{K}\|_{L^1(\mathbb{R})}}.
\]
We also assume that $\theta^{P,n}_{i+1/2}+L$ is nonnegative for any $i\in \I$ and that
\begin{align*}%\label{hypTVBanach}
    TV\left( u^{P,n}\right) \leq TV\left(u^P_0\right) \, e^{10 (L^P)T \, \|\K\|_{L^1(\R)}}.
\end{align*}

\subsubsection{Construction of a contraction mapping}

We introduce the following compact set:
\begin{align*}
    \mathcal{U}_\alpha := \left\{ v=(v_i)_{i \in \I} \,:\,  \max_{i\in\I} \,|v_i|\leq \alpha+1 \right\},
\end{align*}
and, for any $i \in \I$, we define the function $F_i \, : \mathcal{U}_\alpha \, \to \R$ such that
\begin{equation}\begin{array}{ll}\label{defFi}
    F_{i}(v) &= u_i^{P, n} + \frac{\Delta t}{\Delta x}\Big(\lambda^M_i[v]_+ \, (u_{i+1}^{P, n} - \displaystyle u_i^{P, n}) -\lambda^M_i[v]_- \, (u_{i} ^{P, n} - u_{i-1}^{P, n})\Big) +  \Delta t \, L^P \,\lambda^M_i[v], \\
    & \displaystyle = \left(1-\frac{\Delta t}{\Delta x}\left|\lambda^M_i[v]\right|\right) u_i^{P, n} + \frac{\Delta t}{\Delta x}\lambda^M_i[v]_+u_{i+1}^{P, n}+  \frac{\Delta t}{\Delta x}\lambda^M_i[v]_- u_{i-1}^{P, n}+ \Delta t \, L^P \,\lambda^M_i[v]. 
\end{array}\end{equation} 
We notice that the solution, denoted $u^{P,n+1}$, of the scheme~\eqref{schemeIC}--\eqref{h(x)} 
at step $n+1$  is given by
\begin{align*}
    u^{P,n+1}_i = F_i(u^{P,n+1}) \quad \forall i \in \I.
\end{align*}
Let us now prove that the range of $F_i$ is included in $[-\alpha-1,\alpha+1]$. To do so, we first notice, thanks to definition~\eqref{h(x)} and Lemma~\ref{kpdelta less than k}, that for any $v \in \mathcal{U}_\alpha$, we have
\begin{align}
    \left|\lambda_i^M[v]\right| \leq (\alpha+1) \sum_{j \in \I} \Delta x \, \left|\sigma_{M,j}^P\right| \leq 5(\alpha+1) \|\K\|_{L^1(\R)}, \quad \forall i \in \I.
\end{align}
Hence, we conclude using the CFL condition~\eqref{delta t over delta x} that 
\begin{align*}
|F_i(v)| \leq \alpha + \Delta t \, L^P \, 5 \, (\alpha+1) \|\K\|_{L^1(\R)}, \quad \forall i \in \I.
\end{align*}
Thanks to the conditions~\eqref{delta t} we conclude that $F_i(v) \in [-\alpha-1,\alpha+1]$ for any $i \in \I$. We now define the map $F : \mathcal{U}_\alpha \to \mathcal{U}_\alpha$ such that $F(v)=(F_i(v))_{i\in\I}$ where $F_i(v)$ is given by~\eqref{defFi}. Let us prove that $F$ is a contraction mapping on $\mathcal{U}_\alpha$. For this purpose, let $v$ and $w \in \mathcal{U}_\alpha$. Then, for any $i \in \I$, we have
\begin{align*}
    \left|F_i(v) - F_i(w)\right| \leq \left(15\alpha\frac{\Delta t}{\Delta x} \|\K\|_{L^1(\R)}
     + 5\Delta t \, L^P \,  \|\K\|_{L^1(\R)} \right) \, \max_{i \in \I} |v_i-w_i|.
\end{align*}
It remains to apply the conditions~\eqref{delta t} and~\eqref{delta t over delta x} and Banach's fixed point theorem to deduce the existence and uniqueness of $u^{P,n+1}\in \mathcal{U}_\alpha$ solution to the scheme.

\subsubsection{Discrete $L^\infty$ estimate} As a by-product of the previous step, we have
\begin{align*}
    \max_{i \in \I} \, |u^{P,n+1}_i| \leq \alpha+1.
\end{align*}
However, let us show that in fact $u^{P,n+1}$ satisfies the sharpest $L^\infty$ bound~\eqref{Linftybounddis}, i.e., $|u^{P,n+1}_i| \leq \alpha$ for any $i \in \I$. We first rewrite the scheme~\eqref{numerical compacted scheme}, for any $i \in \I$, as
{\small\begin{multline*}
  u_i^{P, n + 1} = 
 \left(1 - \frac{\Delta t}{\Delta x}\left|\lambda^M_i[ u^{P, n + 1}]\right|\right) u_i^{P, n}+  \, \frac{\Delta t}{\Delta x}\lambda^M_i[ u^{P, n + 1}]_+ u_{i + 1}^{P, n}+  \, \frac{\Delta t}{\Delta x}\lambda^M_i[ u^{P, n + 1}]_- u_{i - 1}^{P, n} + \Delta t L^P\lambda^M_i[ u^{P, n + 1}].
\end{multline*}}
Thanks to this decomposition and the CFL condition~\eqref{delta t over delta x} we get
\begin{align*}
    |u^{P,n+1}_i|\leq \max_{i\in \I} \, |u^{P,n}_i|+ \Delta t \, L^P \, \left|\lambda^M_i[u^{P,n+1}]\right|.
\end{align*}
Hence, applying Lemma~\ref{kpdelta less than k} yields
  \begin{align*}
      |u_i^{P, n+ 1}| \leq \max_{i \in \I} \, |u_i^{P, n}| + 5L^P\, \Delta t \,\|\mathcal{K}\|_{L^1(\mathbb{R})} \, \max_{i\in\I} \, |u_i ^{P, n+ 1}|, \quad\mbox{for all}\quad 0\le n \le N_T-1
      \;\;\mbox{and} \;\; i \in \I.
  \end{align*}
Thereby, a direct induction leads to
\begin{align*}
     \max_{i \in \I}\, |u_i^{P, n+ 1}|
     &\leq  \frac{1}{(1 - 5 L^P \Delta t||\mathcal{K}||_{L^1(\mathbb{R})})^{n + 1}}\, \max_{i \in \I} \,|u_i^{P, 0}| \, \leq \, \frac{1}{(1 - 5 L^P \Delta t||\mathcal{K}||_{L^1(\mathbb{R})})^{n + 1}}\, \|u^P_0\|_{L^\infty(I_P)}.
\end{align*}
Therefore, thanks to condition~\eqref{delta t} and the classical inequality $-\ln(1-x)\leq x/(1-x)$ for $0\leq x < 1$, we deduce that $u^{P,n+1}$ satisfies estimate~\eqref{Linftybounddis}.

\subsubsection{Nonnegativity of the discrete gradient} We now prove that $\theta_{i + 1/2}^{P, n + 1} + L^P$ is nonnegative for any $i\in\I$. Using equation~\eqref{numerical compacted scheme}, we first write, for any $i \in \I$,
\begin{multline*}
\theta_{i +1/2}^{P, n + 1} + L^P= \bigg(1 - \frac{\Delta t}{\Delta x} \big( \lambda^M_i[u^{P, n + 1}]_+ +  \lambda^M_{i + 1} [u^{P, n + 1}]_-\big) \bigg) \theta_{i + 1/2}^{P, n } \\
+ \frac{\Delta t}{\Delta x} \lambda^M_{i + 1}[u^{P, n + 1}]_+ \,\theta_{i +3/2}^{P, n } + \frac{\Delta t}{\Delta x}  \lambda^M_i [u^{P, n + 1}]_- \, \theta_{i -1/2}^{P, n } \\ + L^P \bigg(\frac{\Delta t}{\Delta x}( \lambda^M_{i + 1} [u^{P, n + 1}] -  \lambda^M_i [u^{P, n + 1}]) + 1\bigg),
\end{multline*}
or equivalently,
\begin{multline}\label{theta_convexcombin}
\theta_{i +1/2}^{P, n + 1} + L^P=\bigg(1 - \frac{\Delta t}{\Delta x} \big( \lambda^M_i[u^{P, n + 1}]_+ + \lambda^M_{i + 1} [u^{P, n +1}]_-\big) \bigg) \left(\theta_{i +1/2}^{P, n } + L^P\right) \\
+ \frac{\Delta t}{\Delta x} \, \lambda^M_{i + 1} [u^{P, n + 1}]_+ \, \left(\theta_{i +3/2}^{P, n } + L^P\right) +\frac{\Delta t}{\Delta x} \,  \lambda^M_i[u^{P, n + 1}]_- \, \left(\theta_{i -1/2}^{P, n } + L^P\right).
\end{multline}
Therefore, the CFL condition~\eqref{delta t over delta x} implies that $\theta^{P,n+1}_{i+1/2}+L^P$ is a convex combination of nonnegative terms (since we assume by induction that $\theta^{P,n}_{i+1/2}+L^P \geq 0$ for any $i\in\I$). Hence, we conclude that
\begin{align*}
    \theta^{P,n+1}_{i+1/2}+L^P \geq 0,
\end{align*}
for any $i \in \I$.

\subsubsection{Discrete total variation estimate}

Let us now establish the discrete total variation estimate~\eqref{TVestim}. For this purpose, using equation~\eqref{numerical compacted scheme} and the  CFL condition~\eqref{delta t over delta x}, we have
      \begin{multline*}
        TV(u^{P, n+1}) \leq \sum_{i \in\I} \bigg(1 - \frac{\Delta t}{\Delta x}(\lambda^M_i[u^{P, n + 1}]_+ +  \lambda^M_{i + 1} [u^{P, n + 1}]_-)\bigg) \, \left|u_{i + 1}^{P, n} - u_i^{P, n}\right| \\ 
        + \frac{\Delta t}{\Delta x} \sum_{i\in\I} \lambda^M_{i + 1}[u^{P, n + 1}]_+ \, \left|u_{i + 2}^{P, n} - u_{i + 1}^{P, n}\right| + \frac{\Delta t}{\Delta x} \sum_{i\in\I} \lambda^M_i[u^{P, n+ 1}]_- \, \left|u_{i}^{P, n} - u_{i - 1}^{P, n}\right|\\
        + L^P \Delta t \sum_{i\in\I}  \left|\lambda_{i + 1}^M[u^{P, n + 1}] - \lambda^M_i[u^{P, n+ 1}]\right|.
  \end{multline*}
Now, we recombine the terms and we apply Lemma~\ref{kpdelta less than k} to the last term in the right hand side of the previous inequality, and we get
  \begin{multline*}
       TV(u^{P, n+1})\,(1-5L^P \Delta t \, \|\K\|_{L^1(\R)}) \leq TV(u^{P,n}) \\+  \frac{\Delta t}{\Delta x} \sum_{i\in\I} \bigg(\lambda^M_{i + 1} [u^{P, n + 1}]_+\, |u_{i + 2}^{P, n} - u_{i + 1}^{P, n}| - \lambda^M_i[u^{P, n + 1}]_+ \, |u_{i + 1} ^{P, n} - u_i ^{P, n}|\bigg)\\
       +  \frac{\Delta t}{\Delta x} \sum_{i\in\I} \bigg(\lambda^M_i[u^{P, n+ 1}]_- |u_i^{P, n} - u_{i - 1}^{P, n}| - \lambda^M_{i + 1} [u^{P, n + 1}]_-|u_{i + 1}^{P, n} - u_i^{P, n}|\bigg).
  \end{multline*}
Therefore, thanks to the periodicity of the vector $u^{P, n}$, we notice that the two sums in the right hand side vanish so that
\begin{align*}
    TV(u^{P, n+1}) &\leq \frac{1}{1 - 5 L^P \Delta t\, ||\mathcal{K}||_{L^1(\mathbb{R})}} \, TV(u^{P, n}).
\end{align*}
Arguing as in the proof of estimate~\eqref{Linftybounddis} we directly obtain
\begin{align*}
 TV(u^{P, n+1}) \leq TV(u^{P,0}) \, e^{10 (L^P) T ||\mathcal{K}||_{L^1(\mathbb{R})}}.
\end{align*}
Finally, since $TV(u^{P,0}) \leq TV(u^P_0)$, we conclude that the estimate~\eqref{TVestim} holds true. 

\subsubsection{Conclusion} Starting from an initial data $u^P_0$ satisfying the assumptions of Theorem~\ref{thm.scheme}, we deduce by induction the existence and uniqueness of a solution $u^{P,n}$ to the scheme, for any $1\leq n \leq N_T$, satisfying the discrete $L^\infty$ bound \eqref{Linftybounddis} and the discrete total variation estimate~\eqref{TVestim}. Moreover, the vector $(\theta^{P,n}_{i+1/2}+L^P)_{i\in\I}$ is componentwise nonnegative.

\subsection{Discrete gradient entropy estimate}\label{sec.disgradent}

In order to conclude the proof of Theorem~\ref{thm.scheme}, it remains to establish the discrete gradient entropy estimate~\eqref{gradentdis}. For this purpose, we will follow the strategy of~\cite{2}. In particular, our analysis will rely on the two following lemmas.

\begin{lemma}\label{technical} 
Let $\gamma_m > 1$. There exists a nonnegative function $g(\theta, \gamma)$ and a constant $C_{\gamma_m} > 0$ (only depending on $\gamma_m$) such that, for all $\theta > 0$ and $\gamma \in (0, \gamma_m)$, the function $f$ given by~\eqref{1.deff} satisfies
 \begin{equation*}%\label{simple eq}
     f\bigg(\frac{\theta}{\gamma}\bigg) \geq \frac{1}{\gamma}f(\theta) - \frac{1}{\gamma}g(\theta,  \gamma)\ln(\gamma),
 \end{equation*}
 and $$|\theta - g(\theta, \gamma)| \leq C_{\gamma_m} = \frac{\gamma_m - 1}{e\ln(\gamma_m)}.$$
\end{lemma}

\begin{proof}
See Lemma 3.3 in \cite{2}.
\end{proof}

\begin{lemma}\label{convexity inequality of f} 
Let $a_k$ and $\theta_k$ be two finite sequences of nonnegative real numbers such that $0< \sum_k a_k < 2$. Then, defining $\theta = \sum_k a_k \theta_k$, the function $f$ given by~\eqref{1.deff} satisfies the following inequality
\[
f(\theta) \leq \sum_k a_k f(\theta_k) + g\left(\theta,\sum_k a_k \right) \, \ln\left(\sum_k a_k\right),
\]
 where $g(\theta,\gamma)$ is defined in Lemma \ref{technical} for $\gamma_m = 2.$
\end{lemma}

\begin{proof}
See Lemma 3.4 in \cite{2}.
\end{proof}

Thanks to the two above lemmas, we are now in position to establish the discrete entropy inequality~\eqref{gradentdis}. Then, let $0\leq n \leq N_T-1$ be fixed, we define
\[
a_1 := \frac{\Delta t}{\Delta x} \, \lambda^M_{i+1}\left[u^{P,n+1}\right]_+, \quad a_2 := \frac{\Delta t}{\Delta x} \, \lambda^M_i\left[u^{P,n+1}\right]_-,
\]
and
\[
a_3 := 1- \frac{\Delta t}{\Delta x}\left(\lambda^M_{i} \left[u^{P, n+1 }\right]_+ +  \lambda^M_{i + 1} \left[u^{P,n+1}\right]_-\right).
\]
Let us first notice that thanks to the definition~\eqref{h(x)} of $\lambda^M_i$ and the CFL condition~\eqref{delta t over delta x}, the coefficient $a_3>0$. Indeed, for any $i \in \I$, we have
\begin{align*}
\lambda^M_{i} \left[u^{P, n+1 }\right]_+ +  \lambda^M_{i + 1} \left[u^{P,n+1}\right]_- \leq 10 \, \|u^P_0\|_{L^\infty(I_P)} \, e^{10(L^P) T \|\K\|_{L^1(\R)}} \, \|\K\|_{L^1(\R)},
\end{align*}
such that, applying~\eqref{delta t over delta x}, we conclude that $a_3>0$. Now, according to~\eqref{theta_convexcombin} we have
\begin{align*}
    \theta^{P,n+1}_{i+1/2} + L^P = a_1 \left(\theta^{P,n}_{i+3/2} + L^P \right) + a_2 \left( \theta^{P,n}_{i-1/2} + L^P\right) + a_3 \left(\theta^{P,n}_{i+1/2} + L^P\right).
\end{align*}
Moreover, we introduce $\mu^{n+1}_{i+1/2} = 1-(a_1 + a_2 + a_3)$. Then, thanks to the CFL condition~\eqref{delta t over delta x} and arguing as for the sign of $a_3$, we get:
\begin{align}\label{defmu}
1 - \mu_{i +1/2}^{n+1} = a_1+a_2+a_3 = 1 - \frac{\Delta t}{\Delta x}\left(\lambda^M_{i} \left[u^{P, n+1 }\right] -  \lambda^M_{i + 1} \left[u^{P, n+1}\right]\right) \in (0,2).
\end{align}
Hence, applying Lemma~\ref{convexity inequality of f} and the definition of the coefficients $a_1$, $a_2$ and $a_3$, we obtain
\begin{align*}
\sum_{i\in\I} \Delta x\, &f\left(\theta_{i +1/2}^{P, n+1} + L^P\right) \\
&\leq  \sum_{i \in\I} \Delta x \, f\left(\theta_{i + 1/2}^{P, n} + L^P\right) + \sum_{i\in\I} \Delta x \, g\left(\theta_{i +1/2}^{P,n+1} + L^P, 1-\mu^{n+1}_{i+1/2}\right) \,\ln\left(1 - \mu_{i +1/2}^{n+1}\right) \\
&+ \sum_{i\in\I} \Delta t \, \left( \lambda^M_{i + 1} [u^{P, n+1}]_+ \, f\left(\theta_{i +3/2}^{P, n} + L^P\right) - \lambda^M_i[u^{P,n+1}]_+ \, f\left(\theta_{i +1/2}^{P, n } + L^P\right)\right)\\ 
& + \sum_{i \in \I} \Delta t \, \left(\lambda^M_i[u^{P, n+1}]_- \, f\left(\theta_{i -1/2}^{P, n } + L^P\right) - \lambda^M_{i + 1}[u^{P, n+1}]_- \, f\left(\theta_{i +1/2}^{P, n} + L^P\right)\right).
\end{align*}
Thanks to the periodicity of the vectors involved in the two last sums of the right hand side and the inequality $\ln(1-\mu)\leq -\mu$ for all $\mu < 1$, we end up with
\begin{multline}\label{aux1}
\sum_{i\in\I} \Delta x\, f\left(\theta_{i +1/2}^{P, n+1} + L^P\right) 
\leq  \sum_{i \in\I} \Delta x \, f\left(\theta_{i + 1/2}^{P, n} + L^P\right) \\ - \sum_{i\in\I} \Delta x \, g\left(\theta_{i +1/2}^{P,n+1} + L^P, 1-\mu^{n+1}_{i+1/2}\right)  \mu_{i +1/2}^{n+1}.
\end{multline}
Now, we rewrite the last term in the right hand side of the previous inequality as
\begin{align*}
-\sum_{i\in\I} \Delta x \, g\left(\theta_{i +1/2}^{P,n+1} + L^P, 1-\mu^{n+1}_{i+1/2}\right) \, \mu_{i +1/2}^{n+1}  = J_1+J_2+J_3,
\end{align*}
where
\begin{align*}
J_1 &= -\sum_{i\in\I} \Delta x \left( g\left(\theta_{i +1/2}^{P,n+1} + L^P, 1-\mu^{n+1}_{i+1/2}\right) - \left(\theta^{P,n+1}_{i+1/2} + L^P \right)\right) \, \mu_{i +1/2}^{n+1} ,\\
J_2 &= -L^P \sum_{i \in \I} \Delta x \, \mu^{n+1}_{i+1/2},\\
J_3 &= -\sum_{i \in \I} \Delta x \, \theta^{P,n+1}_{i+1/2} \,  \mu^{n+1}_{i+1/2}.
\end{align*}
For the term $J_1$, thanks to Lemma~\ref{technical}, with $\gamma_m=2$, and the definition~\eqref{defmu} of $\mu^{n+1}_{i+1/2}$, we have
\begin{align*}
    |J_1| \leq \frac{1}{e \, \ln(2)} \, \sum_{i \in \I} \Delta x \, \left|\mu^{n+1}_{i+1/2}\right| \leq \frac{1}{e \, \ln(2)} \, \sum_{i \in \I} \Delta t \sum_{j \in \I} \Delta x \left| \sigma^P_{M,j} \left( u^{P,n+1}_{i-j} - u^{P,n+1}_{i+1-j}\right)\right|.
\end{align*}
Therefore, we deduce from the discrete total variation estimate~\eqref{TVestim} and Lemma~\ref{kpdelta less than k}:
\begin{align}\label{J1}
    |J_1| \leq \frac{5 \Delta t \, \|\K\|_{L^1(\R)}}{e \,\ln(2)} TV(u^{P,n+1}) \leq  \frac{5 \Delta t \, \|\K\|_{L^1(\R)} \, e^{10(L^P)T \, \|\K\|_{L^1(\R)}}}{e \, \ln(2)} \, TV(u^P_0).
\end{align}
As a direct by-product of~\eqref{J1}, we also deduce the following estimate
\begin{align}\label{J2}
    |J_2| \leq  5 L^P \, \Delta t \, \|\K\|_{L^1(\R)} \, e^{10(L^P)T \, \|\K\|_{L^1(\R)}} \, TV(u^P_0).
\end{align}
The estimate for $J_3$ is more involved. Indeed, we first rewrite the term $J_3$ thanks to relation~\eqref{defmu} as
\begin{align}\label{toto}
J_3 =  \sum_{i \in \I} \Delta t \,\theta^{P,n+1}_{i+1/2} \sum_{j \in \I} \Delta x \, \sigma^P_{M,j} \left( u^{P,n+1}_{i+1-j} - u^{P,n+1}_{i-j} \right) = \Delta t  \sum_{i \in \I} \Delta t \,\theta^{P,n+1}_{i+1/2} \sum_{j \in \I} \Delta x \, \sigma^P_{M,j} \, \theta^{P,n+1}_{i-j+1/2}
\end{align}
Then, we define the vector
\[
\theta^{P,n+1} = \left(\theta^{P,n+1}_{i+1/2}\right)_{i\in\I},
\]
and, for any vector $v=(v_i)_{i\in\I}$, we introduce the discrete counterpart of the periodic convolution product~\eqref{conv.prod.perio}, still denoted $\ast$, between $v$ and $\theta^{P,n+1}$ as
\begin{align*}
    \left(v \ast \theta^{P,n+1}\right)_i := \sum_{j \in \I} \Delta x \, v_j \, \theta^{P,n+1}_{i-j+1/2} = \sum_{j \in \I}  v_j \, \left(u^{P,n+1}_{i+1-j}-u^{P,n+1}_{i-j}\right), \quad \forall i \in \I.
\end{align*}
Hence, using relation~\eqref{toto}, we notice that we can rewrite $J_3$ as
\begin{align*}
    J_3 = \Delta t \sum_{i\in\I} \Delta x \, \theta^{P,n+1}_{i+1/2} \, \left( \bar{\sigma}_M^P \ast \theta^{P,n+1}\right)_i 
  \quad\mbox{with} \quad    \bar{\sigma}_M^P= (\sigma_{M,i}^P)_{i\in\I}. 
\end{align*}
Now, to deal with the term $J_3$ we will reproduce at the discrete level the computations done in Section~\ref{sec.derivperiodic}. For this, we need to use the Discrete Fourier Transform (DFT) and in the sequel, for any $v = (v_i)_{i\in \I}$, we will denote by $c^d(v) = (c^d_j(v))_{j\in\I}$ the vector of the coefficients of the DFT of $v$ where
\begin{align*}
    c^d_j(v) := \frac{1}{2N} \sum_{\ell \in \I} v_\ell \, e^{-i\pi \, j \, \ell/N}, \quad \forall j \in \I.
\end{align*}
Therefore, applying the discrete version of Parseval's equality as well as the property of the DFT with respect to the discrete convolution product, we get
\begin{align*}
    J_3 = (2N)^2 \, \Delta t  \, \sum_{j \in \I} \Delta x \, c^d_j\left(\bar{\sigma}_M^P\right) \, \left|c^d_j\left(\theta^{P,n+1}\right)\right|^2.
\end{align*}
Let us now show that the coefficients of the DFT of the vector $\bar{\sigma}_M^P$ are real nonpositive numbers. For this purpose, let $j\in\I$ be fixed. Then, by definition we have
\begin{align*}
    c^d_j\left(\bar{\sigma}_M^P\right) = \frac{1}{2N} \sum_{\ell \in \I} \sigma_{M,\ell}^P \, e^{-i\pi \, j \, \ell/N} &= \frac{1}{2NM} \sum_{\ell\in\I} \sum_{k=0}^{M-1} \sum_{|m|\leq k} c_m(\K^P_M) \, e^{i\pi (m-j) \ell/N}\\
    &=\frac{1}{2NM} \sum_{k=0}^{M-1} \sum_{|m|\leq k} c_m(\K^P_M) \, \sum_{\ell\in\I} e^{i\pi (m-j) \ell/N}.
\end{align*}
Hence, since the last sum is either equal to $0$ or $2N$, depending on the values of $m$, we conclude that $c^d_j(\bar{\sigma}_M^P)$ is a sum of the Fourier coefficients $c_m(\K^P_M)$ for $|m|<M$. Therefore, thanks to Lemma~\ref{fourier series of kpdelta negative}, we obtain
\begin{align}\label{J3}
    J_3 \leq 0.
\end{align}
Thereby, gathering~\eqref{aux1}--\eqref{J3}, we have
\begin{multline}\label{gradentdis.onesteptime}
    \sum_{i\in\I} \Delta x\, f\left(\theta_{i +1/2}^{P, n+1} + L^P\right) 
\leq  \sum_{i \in\I} \Delta x \, f\left(\theta_{i + 1/2}^{P, n} + L^P\right) \\
+  5 \, \Delta t \, \|\K\|_{L^1(\R)} \, e^{10(L^P)T \, \|\K\|_{L^1(\R)}} \, TV(u^P_0) \left(\frac{1}{e \, \ln(2)} + L^P \right).
\end{multline}
Summing the previous inequality over $n$ yields the existence of the constant $\zeta>0$, given by
\begin{align*}
    \zeta := 5T \, \|\K\|_{L^1(\R)} \, e^{10 (L^P) T \, \|\K\|_{L^1(\R)}} \, \left(\frac{1}{e \ln(2)}+L^P\right),
\end{align*}
such that
$$
   \sum_{i\in\I} \Delta x\, f\left(\theta_{i +1/2}^{P, n+1} + L^P\right) 
\leq  \sum_{i \in\I} \Delta x \, f\left(\theta_{i + 1/2}^{P, 0} + L^P\right) +  \zeta TV(u^P_0).
$$
It is quite clear that in the last estimate  $\zeta$  does not depend on $M$. 
Moreover, since we have $L^P = \frac{v_0(P)-v_0(-P)}{2P}$ (cf. \eqref{def.LP}),  we can conclude that $\zeta$ is uniformly bounded with respect to $P$, when $P\ge 1$, namely
\begin{align}\label{defzeta}
    \zeta \leq 5T \, \|\K\|_{L^1(\R)} \, e^{10 T \,\|v_0\|_{L^\infty(\R)} \, \|\K\|_{L^1(\R)}} \, \left(\frac{1}{e \ln(2)}+\|v_0\|_{L^\infty(\R)}\right) :=  \zeta_0.
\end{align}
This implies \eqref{gradentdis} and concludes the the proof of Theorem~\ref{thm.scheme}.

\begin{remark}\label{rem.alter2}
Let us now explain how to establish a discrete gradient entropy estimate with the definition~\eqref{def.alternative} of $\sigma^P_{M,j}$. In this case, the estimate of the term $J_3$ is more involved. Indeed, we write
\begin{align*}
    J_3 = \Delta t \sum_{i\in\I} \Delta x \, \theta^{P,n+1}_{i+1/2} \, \left( (\widetilde{\sigma}_M^P-\bar{\sigma}^P_M) \ast \theta^{P,n+1}\right)_i +\Delta t \sum_{i\in\I} \Delta x \, \theta^{P,n+1}_{i+1/2} \, \left(\bar{\sigma}^P_M \ast \theta^{P,n+1}\right)_i =: J_{31}+J_{32},
\end{align*}
with
\begin{align*}
\bar{\sigma}_M^P= (\sigma_M^P(x_j))_{j\in\I}, \quad \widetilde{\sigma}_M^P = (\sigma_{M,j}^P)_{j \in \I}.
\end{align*}
where obviously $\sigma^P_{M,j}$ is given by~\eqref{def.alternative}. Then, the term $J_{32}$ is estimated as before and for the term $J_{31}$ we have
\begin{multline*}
|J_{31}| \leq \Delta t  \sum_{i \in \I} \sum_{j\in\I} (\Delta x)^2 \, \left|\sigma^P_{M,j}- \sigma^P_M(x_j)\right|  \, \left|\theta^{P,n+1}_{i+1/2}\right| \, \left|\theta^{P,n+1}_{i-j+1/2}\right|\\ \leq \Delta t \, \left(TV\left(u^{P,n+1}\right)\right)^2 \, \max_{j \in \I} \left|\sigma^P_{M,j}- \sigma^P_M(x_j)\right|.%\le C_M^P \Delta t \Delta x. 
\end{multline*}
Moreover, for any $j \in \I$, it holds
\begin{align*}
&\left|\sigma^P_{M,j}- \sigma^P_M(x_j)\right|
=   \left| \frac{1}{2P} \sum_{i\in\I} \Delta x\,  F_M(x_j-x_i) \,\K^P_{M,i} -  
\frac{1}{2P} \int_{-P}^{P} F_M(x_j-x)\,\K^P_M(x) \,  \dd x,\right|\\
\phantom{xx}&= \frac{1}{2P} \left| \sum_{i\in\I} \int_{x_{i-1/2}}^{x_{i+1/2}} F_M(x_j-x_i) \,\K^P_{M}(x)\,  \dd x      
- \sum_{i\in\I}\int_{x_{i-1/2}}^{x_{i+1/2}} F_M(x_j-x)\, \K^P_{M}(x) \, \dd x   \right| \\
\phantom{xx}&= \frac{1}{2P}  \left| \sum_{i\in\I} \int_{x_{i-1/2}}^{x_{i+1/2}} (F_M(x_j-x_i)-F_M(x_j-x))\K^P_{M}(x) \dd x       \right|
\le \frac{5 \, \Delta x}{4P} \|\K\|_{L^1(\R)}\,\|\pa_x F_M\|_{L^\infty(I_P)}. 
\end{align*}
Therefore, there exists a constant $C_M^P$ such that
\begin{align}\label{est_N_NT}
|J_{31}| \leq C_M^P \, \Delta t \, \Delta x.
\end{align}
Since we will first pass to the limit $\Delta x\to 0$ for the proof of Theorem~\ref{convergence in periodic}, then the dependency on $P$ and $M$ of the constant $C_M^P$ has no consequence for the proof of Proposition~\ref{prop.PMinfty}. 
\end{remark}

\section{Convergence of the scheme}\label{convergence}
In this section, we prove Theorem~\ref{convergence in periodic}. For this purpose, we split our proof in three main steps. First, we establish some uniform in $\eps=(\Delta x, \Delta t)$ estimates. Then, we show the existence of a function $u^P_M$ such that, up to a subsequence, the sequence $(u^{P,m}_M)_{m \in \N}$ converges toward $u^P_M$ (in the sense specify in the statement of Theorem~\ref{convergence in periodic}). Finally, we identify this function as a solution of~\eqref{1.equPaux} in the distributional sense. Our convergence proof is similar to the one of~\cite{2} and relies on three technical results stated in Appendix~\ref{app.techni}.

\subsection{Uniform estimates}\label{sec.unifestimates}

In this section, for fixed values of $\Delta x$ and $\Delta t$, satisfying condition~\eqref{delta t} and~\eqref{delta t over delta x}, we establish the following:

\begin{proposition}\label{prop.unif}
Let the assumptions of Theorem~\ref{thm.scheme} hold. Then, there exists a constant ${C}_{\rm GE}>0$, only depending on $v_0$, $T$ and $\K$, such that the function $u^{P,\eps}_M$, defined by~\eqref{defQ1ext}, satisfies
\begin{align}\label{estimationLlogLGE}
    \|\pa_x u^{P,\eps}_M\|_{L^\infty(0,T;L\,\log\,L(I_P))} + \|\pa_t u^{P,\eps}_M\|_{L^\infty(0,T;L\,\log\,L(I_P))} \leq {C}_{\rm GE}.
\end{align}
\end{proposition}

The main idea to prove Proposition~\ref{prop.unif}, following for instance~\cite{2}, is to apply estimates~\eqref{f in llogl}-\eqref{w in llogl} of Lemma~\ref{Llog L estimate}. To do so, roughly speaking, we have to establish an uniform $L^\infty(0,T;L^1(I_P))$ estimates on $\pa_x u^{P,\eps}_M$ and $\pa_t u^{P,\eps}_M$, as well as, some $\eps$-uniform bounds on
\begin{align*}%\label{goalunif}
    \int_{I_P} f\left(\pa_x u^{P,\eps}_M(x,t)+L^P\right) \, \dd x, \quad \int_{I_P} f\left(\left|\pa_t u^{P,\eps}_M(x,t)\right|\right) \, \dd x, \quad \mbox{for a.e. }t\in (0,T).
\end{align*}
While the $L^\infty(0,T;L^1(I_P))$ estimates will be consequences of~\eqref{Linftybounddis}--\eqref{TVestim} and the definition~\eqref{defQ1ext} of $u^{P,\eps}_M$, the former bounds are consequences of the discrete gradient entropy estimate~\eqref{gradentdis}. However, let us notice that the first term in the right hand side of~\eqref{gradentdis} depends on $\Delta x$ through the term $f(\theta^{P,0}_{i+1/2}+L^P)$. Therefore, we need to establish a uniform~w.r.t. $\eps$ estimate on this term.

\begin{lemma}\label{lem.boundinitent}
Let the assumptions of Theorem~\ref{thm.scheme} hold. Then, there exists a constant $C_0>0$ 
only depending on $v_0$ 
 such that
\begin{align*}
    I_0=\sum_{i\in\I} \Delta x \, f\left(\theta^{P,0}_{i+1/2}+L^P\right) \leq C_0, 
\end{align*}
where $L^P$ is defined in \eqref{def.LP}. 

\end{lemma}

\begin{proof}
Thanks to the convexity of $f$, we have, for any $i \in \I$,
\begin{align*}
f\left(\theta_{i+\frac{1}{2}} ^{p, 0} + L^P\right) =  f\left(\frac{1}{\Delta x}\int_{x_i} ^{x_{i+1}}\left[\partial_x u_0^P (y)+ L^P\right] \, \dd y\right)\leq \frac{1}{\Delta x}\int_{x_i} ^{x_{i+1}} f\left(\partial_x u_0^P (y) + L^P\right) \, \dd y.
\end{align*}
Applying~\eqref{f in llogl} in Lemma~\ref{Llog L estimate} yields
\begin{align*}
I_0 &\leq 1 + \|\pa_x u^P_0+ L^P\|_{L\,\log\,L(I_P)} + \left(\|\pa_x u^P_0\|_{L^1(I_P)} + 2PL^P\right) \, \ln\left(1+\|\pa_x u^P_0+L^P\|_{L\,\log\,L(I_P)}\right), \\
&\leq 1 + \|\pa_x u^P_0\|_{L\,\log\,L(I_P)}+ \|L^P\|_{L\,\log\,L(I_P)} \\
&\phantom{xxxxxxxxxxxxxx}+ \left(\|\pa_x u^P_0\|_{L^1(I_P)} + 2PL^P\right) \, \ln\left(1+\|\pa_x u^P_0\|_{L\,\log\,L(I_P)}+\|L^P\|_{L\,\log\,L(I_P)}\right).     
\end{align*}
Thanks to the standard properties of the norms and \eqref{w in llogl} in Lemma~\ref{Llog L estimate} we can see that 

\begin{equation}\label{estimL.LlogL}\begin{array}{ll}
\|L^P\|_{L\,\log\,L(I_P)}=L^P\|1\|_{L\,\log\,L(I_P)}  &
\displaystyle\leq L^P\left(1+ 2P \, \ln(1+e^2) + \int_{I_P} f(1)\, \dd x\right), 
\\ &
\displaystyle \leq \|v_0\|_{L^\infty(\R)}\left(1+2\ln(1+e^2)+\frac{2}{e}\right). 
\end{array}\end{equation}
Moreover Lemma~\ref{lem.init.perio.to.nonperiod} (see~\eqref{estimationL1P}-\eqref{estimationLlogLP}) states that 
$\|\pa_x u^P_0\|_{L\,\log\,L(I_P)}$ and $\|\pa_x u^P_0\|_{L^1(I_P)}$ are uniformly bounded with respect to $P$ and $M$. This fact, joins to~\eqref{estimL.LlogL}  
proves that $I_0$ is  uniformly bounded (with respect to $P$ and $M$) and completes the proof of this lemma. 
\end{proof}

We are now in position to prove Proposition~\ref{prop.unif}. For this purpose, we will split our proof in several steps.

 \subsubsection{$L^\infty(0,T;L^1(I_P))$ bounds on the gradient and time derivative of $u^{P,\eps}_M$}

\begin{lemma}\label{lem.LinftyL1estime}
Let the assumptions of Theorem~\ref{thm.scheme} hold. Then, there exists a constant $C_1>0$ only depending on $v_0$, $T$ and $\K$ such that
\begin{align}\label{LinftyL1.paxu.patu}
    \|\pa_x u^{P,\eps}_M\|_{L^\infty(0,T;L^1(I_P))} + \|\pa_t u^{P,\eps}_M\|_{L^\infty(0,T;L^1(I_P))} \leq C_1.
\end{align}
\end{lemma}

\begin{proof}
We first notice, thanks to the definition~\eqref{defQ1ext} of $u^{P,\eps}_M$, that for any $(x,t) \in (x_i,x_{i+1}) \times [t_n , t_{n+1}]$ we have
\begin{equation}\label{uepsilonx}
\partial_x u^{P, \eps}_{M} (x,t) = \bigg(\frac{t - t_n}{\Delta t}\bigg) \,\theta_{i +1/2}^{P, n + 1} + \bigg(1 - \frac{t - t_n}{\Delta t}\bigg) \, \theta_{i +1/2}^{P, n }.
\end{equation}
Therefore, we obtain for a given $t \in [t_n, t_{n + 1}]$ (for some $0\leq n \leq N_T-1)$
\begin{multline*}
\int_{I_P} \left|\partial_x u^{P, \eps}_{M}(x,t)\right|\, \dd x = \sum_{i \in \I} \int_{x_i}^{x_{i + 1}} \left|\partial_x u^{P,\eps}_M(x,t)\right|\, \dd x\\ 
\leq \sum_{i\in \I}  \Delta x\, \bigg(\frac{t - t_n}{\Delta t}\bigg)\, \left|\theta_{i +1/2}^{P, n + 1}\right| + \sum_{i\in \I} \Delta x \, \bigg(1 - \frac{t - t_n}{\Delta t}\bigg) \, \left|\theta_{i +1/2}^{P, n }\right|,
\end{multline*}
so that
\begin{multline*}
     \int_{I_P} \left|\partial_x u^{P, \eps}_{M}(x,t)\right|\, \dd x \leq \sum_{i\in \I} \Delta x \, \left|\theta_{i +1/2}^{P, n + 1}\right| + \sum_{i\in \I} \Delta x \, \left|\theta_{i +1/2}^{P, n }\right| \\
     = \sum_{i\in\I} \left(\left|u^{P,n+1}_{i+1}-u^{P,n+1}_i\right|+\left|u^{P,n}_{i+1}-u^{P,n}_i\right|\right).
\end{multline*}
Applying estimate~\eqref{TVestim} in Theorem~\ref{thm.scheme}, we end up with
\begin{equation}\label{bounded ux}
\|\partial_x u^{P, \eps}_M\|_{L^{\infty}(0, T; L^1(I_P))} \leq 2\, TV\left(u^P_0\right) \,e^{10(L^P) \, T \, \|\K\|_{L^1(\mathbb{R})}}.
\end{equation}
Now, let us define:
\begin{align}\label{def.tau}
\tau_i^{P, n +1/2} := \frac{u_i^{P, n + 1} - u_i^{P, n}}{\Delta t}, \quad \forall i \in \I, \, 0\leq n \leq N_T-1. 
\end{align}
Then, for $(x,t) \in [x_i,x_{i+1}] \times (t_n, t_{n + 1})$, we have
\begin{equation}\label{decompo.patuPeps}
\partial_t u^{P, \eps}_M(x,t) = \bigg(\frac{x - x_i}{\Delta x}\bigg) \, \tau_{i + 1}^{P, n +1/2}+\bigg( 1 - \frac{x - x_i}{\Delta x}\bigg)\, \tau_i^{P, n +1/2}. 
\end{equation}
Therefore, it holds
\begin{align*}
\int_{I_P} \left|\partial_t u^{P, \eps}_M(x,t)\right| \,\dd x \leq \sum_{i\in\I} \int_{x_i}^{x_{i+1}} \left( \bigg(\frac{x - x_i}{\Delta x}\bigg) \, \left|\tau_{i + 1}^{P, n +1/2}\right|  + \bigg( 1 - \frac{x - x_i}{\Delta x}\bigg) \,\left|\tau_i^{P, n +1/2} \right| \right)\, \dd x,
\end{align*}
so that
\begin{equation*}
\int_{I_P} \left|\partial_t u^{P, \eps}_M(x,t)\right| \,\dd x \leq \frac12 \sum_{i\in\I} \Delta x \, \left(\left|\tau_{i}^{P, n +1/2}\right|+\left|\tau_{i + 1}^{P, n +1/2}\right|\right)=: \frac12(J_4+J_5).
\end{equation*}
Using the definition of the scheme, we notice that
\begin{equation}\label{tau in theta}
\tau_i^{P, n +1/2} = \lambda_i[u^{P, n + 1}]_+ \, \left(\theta_{i +1/2}^{P,n} + L^P\right) - \lambda_i[u^{P,n + 1}]_- \, \left(\theta_{i-1/2}^{P, n} + L^P\right).
\end{equation}
Hence, thanks to Lemma~\ref{kpdelta less than k}, the $L^\infty$ bound~\eqref{Linftybounddis} and the discrete total variation estimate~\eqref{TVestim}, we obtain
\begin{align*}
J_4 &\leq \sum_{i\in\I} \lambda_i[u^{P,n+1}]_+ \, \left(\left|u^{P,n}_{i+1}-u^{P,n}_i\right|+ L^P \, \Delta x\right) + \sum_{i\in\I} \lambda_i[u^{P,n+1}]_- \, \left(\left|u^{P,n}_i-u^{P,n}_{i-1}\right| + L^P \, \Delta x \right)\\
&\leq 5 \|u^P_0\|_{L^\infty(I_P)} \, \|\K\|_{L^1(\R)} \, e^{10(L^P)T \, \|\K\|_{L^1(\R)}} \, \left(4PL^P + \sum_{i\in\I} \left(\left|u^{P,n}_{i+1}-u^{P,n}_i\right|+\left|u^{P,n}_i-u^{P,n}_{i-1}\right|\right)\right)\\
&\leq 10 \|u^P_0\|_{L^\infty(I_P)} \, \|\K\|_{L^1(\R)} \, e^{10(L^P)T \, \|\K\|_{L^1(\R)}} \, \left(2PL^P + TV(u^P_0) \, e^{10(L^P)T \, \|\K\|_{L^1(\R)}} \right).
\end{align*}
We readily obtain a similar bound for the term $J_5$, such that
\begin{equation}\begin{array}{ll}\label{ut}
\displaystyle\int_{I_P} \left|\partial_t u^{P, \eps}_M(x,t)\right| \,\dd x \\\leq  10 \|u^P_0\|_{L^\infty(I_P)} \, \|\K\|_{L^1(\R)} \, e^{10(L^P)T \, \|\K\|_{L^1(\R)}} \, \left(2PL^P + TV(u^P_0) \, e^{10(L^P)T \, \|\K\|_{L^1(\R)}} \right).
\end{array}\end{equation}
Moreover, from the definition of $L^P$ in 
\eqref{def.LP}, we know  that (for $P\ge 1$)
\begin{equation}\label{estimationLPPLP}L^P  \leq \|v_0\|_{L^\infty(\R)},  
\quad \mbox{and}\quad PL^P  \leq \|v_0\|_{L^\infty(\R)} \end{equation}
as well as, from Lemma \ref{lem.init.perio.to.nonperiod} (see \eqref{estimationLinP}-\eqref{estimationL1P}), we have 
$ \|u^P_0\|_{L^\infty(I_P)}$ and $TV(u^P_0)$ are uniformly bounded with respect to $P$ and $M$. 
This involves, gathering~\eqref{bounded ux} and~\eqref{ut}, that there exists  a constant
$C_1 >0$ independent of $P$ and $M$ such that~\eqref{LinftyL1.paxu.patu} holds.
\end{proof}

\subsubsection{Gradient entropy estimates}

\begin{lemma}\label{lem.fpartial}
Let the assumptions of Theorem~\ref{thm.scheme} hold. Then, there exists a constant $C_2>0$ only depending on  $v_0$, $T$ and $\K$ such that
\begin{align}
    \int_{I_P} f\left(\pa_x u^{P,\eps}_M(x,t)+L^P \right) \, \dd x + \int_{I_P} f\left(\left|\pa_t u^{P,\eps}_M(x,t)\right| \right) \, \dd x \leq C_2, \quad \mbox{for a.e. }t\in(0,T).
\end{align}
\end{lemma}

\begin{proof}
First, for $(x,t)\in(x_i,x_{i+1})\times[t_n,t_{n+1}]$ with $i\in \I$ and $0\le n \le N_T-1$, we write
\begin{equation}\label{uepsilonx + L}
\partial_x u^{P, \eps}_M(x,t)  + L^P=  \bigg(\frac{t - t_n}{\Delta t}\bigg) \, \left(\theta_{i +1/2}^{P, n + 1} + L^P\right) + \bigg(1 - \frac{t - t_n}{\Delta t}\bigg) \, \left(\theta_{i +1/2}^{P, n }+ L^P\right). 
\end{equation}
Now, using the convexity of $f$, we obtain
\[
f\left(\partial_x u^{P, \eps}_M(x,t)+L^P \right) \leq \bigg(\frac{t - t_n}{\Delta t}\bigg) \, f\left(\theta_{i +1/2} ^{P, n + 1} + L^P\right)+ \bigg(1 - \frac{t - t_n}{\Delta t}\bigg) \, f\left(\theta_{i +1/2}^{P, n }+ L^P\right).
\]
Therefore, by integrating in space, we have
\begin{multline*}
\int_{I_P} f\left(\partial_x u^{P, \eps}_M(x,t)+L^P \right)\, \dd x \leq \bigg(\frac{t - t_n}{\Delta t}\bigg) \, \sum_{i\in\I} \Delta x\, f\left(\theta^{P,n+1}_{i+1/2}+L^P\right)\\
+ \bigg(1 - \frac{t - t_n}{\Delta t}\bigg) \, \sum_{i\in\I} \Delta x \, f\left(\theta^{P,n}_{i+1/2}+L^P\right).
\end{multline*}
Thanks to estimate~\eqref{gradentdis} we get
\begin{align*}
\int_{I_P} f\left(\partial_x u^{P, \eps}_M(x,t)+L^P \right)\, \dd x &\leq \sum_{i\in\I} \Delta x \, f\left(\theta^{P,0}_{i+1/2}+L^P\right) + \zeta_0 TV(u^P_0), 
\end{align*}
where $\zeta_0$ is defined~\eqref{defzeta}. Applying Lemma \ref{lem.boundinitent}, we deduce that 
\begin{align*}
   \int_{I_P} f\left(\partial_x u^{P, \eps}_M(x,t)+L^P \right)\, \dd x 
   \leq C_0 + \zeta_0  TV(u^P_0)
   \le C_0 + 4\zeta_0  \|v_0\|_{L^\infty(\R)}
\end{align*}
where for the last inequality we have used Lemma~ \ref{lem.init.perio.to.nonperiod} (cf. \eqref{estimationL1P}).\\
 Let us now establish the estimate for $\pa_t u^{P,\eps}_M$. For this purpose, using~\eqref{decompo.patuPeps}, we notice that for all $(x,t) \in [x_i,x_{i+1}]\times(t_n,t_{n+1})$
\begin{align*}
      \left|\partial_t u^{P, \eps}_M(x,t)\right| &\leq \bigg(\frac{x - x_i}{\Delta x}\bigg) \,\left|\tau_{i + 1}^{P, n + 1/2}\right| + \bigg( 1 - \frac{x - x_i}{\Delta x}\bigg) \, \left|\tau_{i } ^{P, n +1/2}\right|,
\end{align*}
where we recall definition~\eqref{def.tau} of $\tau^{P,n+1/2}_i$. Hence, the convexity of $f$ yields
\begin{multline*}
    f\left(\left|\partial_t u^{p, \eps}_M(x,t)\right|\right) \leq \bigg(\frac{x - x_i}{\Delta x}\bigg)\,f\left(\left|\tau_{i + 1}^{P, n + 1/2}\right|\right) + \bigg( 1 - \frac{x - x_i}{\Delta x}\bigg) \, f\left(\left|\tau_{i } ^{P, n +1/2}\right|\right)\\
    \leq f\left(\left|\tau_{i + 1}^{P, n + 1/2}\right|\right) +  f\left(\left|\tau_{i } ^{P, n +1/2}\right|\right).
\end{multline*}
Let us for instance consider the second term in the right hand side. Then, according to the relation~\eqref{tau in theta}, we have
\begin{align*}
    \left|\tau_i^{P, n +1/2}\right| \leq \lambda_i[u^{P, n + 1}]_+ \, \left(\theta_{i +1/2}^{P,n} + L^P\right) + \lambda_i[u^{P,n + 1}]_- \, \left(\theta_{i-1/2}^{P, n} + L^P\right).
\end{align*}
Now, let us assume that $\lambda_i[u^{P,n+1}]\ge 0$. This yields
\begin{align*}
    \left|\tau_i^{P, n +1/2}\right| \leq |\lambda_i[u^{P, n + 1}]| \, \left(\theta_{i +1/2}^{P,n} + L^P\right).
\end{align*}
Then, we notice thanks to the $L^\infty$ estimate~\eqref{Linftybounddis} and the discrete total variation estimate~\eqref{TVestim}, that it holds
\begin{align*}
    \left|\tau_i^{P, n +1/2}\right| \leq \Lambda\, \left(\theta_{i +1/2}^{P,n} + L^P\right),
\end{align*}
where
\begin{align*}
    \Lambda := \max\left(1,5\|u^P_0\|_{L^\infty(I_P)} \, \|\K\|_{L^1(\R)} \, e^{10 (L^P)T \, \|\K\|_{L^1(\R)}}\right).
\end{align*}
Hence, using the elementary estimate $f(\gamma \theta) \leq \gamma f(\theta)+\theta f(\gamma)$ for all $\gamma\geq 1$ and $\theta \geq 0$, we have
\begin{align*}
    f\left(\left|\tau_i^{P, n +1/2}\right|\right) \leq \Lambda\, f\left(\theta_{i +1/2}^{P,n} + L^P\right) + \left(\theta_{i +1/2}^{P,n} + L^P\right)  \, f\left(\Lambda\right).
\end{align*}
Of course, one obtain a similar bound if $\lambda_i[u^{P,n+1}]\le 0$, so that
\begin{align}\label{bound.ftau}
    f\left(\left|\tau_i^{P, n +1/2}\right|\right) &\leq \sum_{\pm} \left[\Lambda\, f\left(\theta_{i \pm 1/2}^{P,n} + L^P\right) + \left(\theta_{i \pm 1/2}^{P,n} + L^P\right)  \, f\left(\Lambda\right) \right].
\end{align}
Thereby, we end up with
\begin{align*}
  I_1=  \int_{I_P} &f\left(\left|\pa_t u^{P,\eps}_M(x,t)\right|\right) \, \dd x \leq \sum_{i\in\I} \Delta x \, f\left(\left|\tau_{i + 1}^{P, n + 1/2}\right|\right) + \sum_{i\in\I} \Delta x\,  f\left(\left|\tau_{i } ^{P, n +1/2}\right|\right)\\
    &\leq 4\Lambda \left(C_0 + \zeta_0 \,TV(u^P_0)\right) + 8 PL^P \, f(\Lambda) + 4f(\Lambda) \, TV(u^P_0) \, e^{10 (L^P)T \, \|K\|_{L^1(\R)}},
\end{align*}
where we have used inequality~\eqref{bound.ftau}, the discrete gradient entropy inequality~\eqref{gradentdis} togther with Lemma~\ref{lem.boundinitent}, and the discrete total variation estimate~\eqref{TVestim}.\\
To complete the proof, we proceed as in the proof of Lemma \ref{lem.LinftyL1estime}. Indeed, we use \eqref{estimationLPPLP}, Lemma \ref{lem.init.perio.to.nonperiod} \eqref{estimationLinP}-\eqref{estimationL1P} and Lemma~\ref{lem.boundinitent},  in order to bound from above $I_1$ by a constant  $C_2$  independent of $M$ and $P$. 
\end{proof}

\subsubsection{Proof of Proposition~\ref{prop.unif}}

Thanks to Lemma~\ref{Llog L estimate} we have
\begin{multline*}
    \|\pa_x u^{P,\eps}_M+L^P\|_{L^\infty(0,T;L\,\log\,L(I_P))} \\ \leq 1 + \left(\|\pa_x u^{P,\eps}_M\|_{L^\infty(0,T;L^1(I_P))} + 2P L^P\right) \, \ln(1+e^2) + \sup_{t \in (0,T)} \int_{I_P} f\left(\pa_x u^{P,\eps}_M(x,t)+L^P\right) \, \dd x,
\end{multline*}
so that, arguing as in~\eqref{estimL.LlogL}, we have
\begin{align}\label{aux.conv}
    \|\pa_x u^{P,\eps}_M\|_{L^\infty(0,T;L\,\log\,L(I_P))} &\leq L^P\|1\|_{L^\infty(0,T;L\,\log\,L(I_P))} + \|\pa_x u^{P,\eps}_M+L\|_{L^\infty(0,T;L\,\log\,L(I_P))}\\
    &\leq  \|v_0\|_{L^\infty(\R)}\left(1+2\ln(1+e^2)+\frac{2}{e}\right) + \|\pa_x u^{P,\eps}_M+L\|_{L^\infty(0,T;L\,\log\,L(I_P))}.\nonumber
\end{align}
and
\begin{multline*}
    \|\pa_t u^{P,\eps}_M\|_{L^\infty(0,T;L\,\log\,L(I_P))}  \leq 1 + \|\pa_t u^{P,\eps}_M\|_{L^\infty(0,T;L^1(I_P))}  \, \ln(1+e^2) + \sup_{t \in (0,T)} \int_{I_P} f\left(\left|\pa_t u^{P,\eps}_M(x,t)\right|\right) \, \dd x.
\end{multline*}
Therefore, according to, \eqref{estimationLPPLP}, Lemma~\ref{lem.LinftyL1estime} and Lemma~\ref{lem.fpartial}, we deduce the existence of a constant ${C}_{\rm GE}>0$ only depending on $v_0$, $T$ and $\K$ such that
\begin{align*}
    \|\pa_x u^{P,\eps}_M\|_{L^\infty(0,T;L\,\log\,L(I_P))} + \|\pa_t u^{P,\eps}_M\|_{L^\infty(0,T;L\,\log\,L(I_P))} \leq {C}_{\rm GE}.
\end{align*}
This concludes the proof of Proposition~\ref{prop.unif}.

\subsection{Compactness properties}

\begin{proposition}\label{prop.compactness}
Let the assumptions of Theorem~\ref{convergence in periodic} hold. Then, there exists a function $u^P_M$ satisfying properties~\eqref{first property}--\eqref{third property} only depending on $v_0$, $T$ and $\K$,such that the family $(u^{P,m}_M)_{m\in \N}$, solution to the scheme~\eqref{schemeIC}--\eqref{h(x)}, converges uniformly, up to a subsequence, to $u^P_M$ as $m\to+\infty$. Moreover, $u^P_M \in C(I_P\times[0,T))$ and there exists a modulus of continuity $\omega(\gamma,h)$ such that for all $\gamma, \,h\geq 0$ and all $(x,t) \in I_P \times (0,T-h)$ we have
    \begin{align}\label{mod-continuite}
|u^P_M (x + \gamma,t+h) - u^P_M (x,t)| \leq 6 \,C_{\rm GE}\, \omega(\gamma,h), \quad \mbox{with } \omega(\gamma, h) =\frac{1}{\ln(1+ \frac{1}{\gamma})}+ \frac{1}{\ln(1 + \frac{1}{h})},
\end{align}
where $C_{\rm GE}>0$ is the constant appearing in~\eqref{estimationLlogLGE}.
\end{proposition}

\begin{proof}

Thanks to Proposition~\ref{prop.unif}-\eqref{estimationLlogLGE} and Lemma~\ref{modulus of continuity}, we obtain, for any $(x,t) \in I_P \times (0,T-\tau)$, the uniform w.r.t.~$\eps_m$ estimate
\begin{align}\label{mod.cont.dis}
    \left|u^{P,m}_M(x+\gamma,t+h)-u^{P,m}_M(x,t)\right| \leq 6\, {C}_{\rm GE} \, \left(\frac{1}{\ln(1+\frac{1}{\gamma})}+\frac{1}{\ln(1+\frac{1}{h})}\right).
\end{align}
Hence, applying Arzel\'a-Ascoli theorem we deduce the existence of a function $u^P_M \in C(I_P\times [0,T))$, such that the family $(u^{P,m}_M)_{m\in\N}$ converges uniformly, up to a subsequence, to $u^P_M$. Moreover, from \eqref{Linftybounddis} we can check that,  the function $u^P_M$ satisfies the $L^\infty$ bound~\eqref{second property} and the initial condition:
\begin{align*}
    u^P_M(\cdot,0) = u^P_0(\cdot).
\end{align*}
Furthermore, as a consequence of estimate \eqref{estimationLlogLGE}, we can say that 
$(\pa_x u^{P,m}_M)_{m\in\N}$ and $(\pa_t u^{P,m}_M)_{m\in\N}$ converge respectively towards 
$\pa_x u^{P}_M$ and $\pa_t u^{P}_M$ weakly-$\star$ in $L^\infty(0,T;L\,\log\,L(I_P))$. 
Therefore, by the lower semi-continuity of the norm for weak-$\star$  topology we have 
$$
   \|\partial_x u^P_M \|_{L^{\infty}(0, T; L \log L (I_P) )} + \|\partial_t u^P_M \|_{L^{\infty}(0, T; L \log L (I_P) )} \leq C_{\rm GE}.$$
   This implies that \eqref{first property} and~\eqref{third property} hold. 
Using the Lemma~\ref{modulus of continuity}  again we prove \eqref{mod-continuite} 
and conclude the proof of Proposition~\ref{prop.compactness}.
\end{proof}
\subsection{Identification of the limit}\label{sec.identi}
In order to prove Theorem~\ref{convergence in periodic}, it remains to identify the function $u^P_M$ obtained in Proposition~\ref{prop.compactness} as a solution to~\eqref{1.equPaux} in the distributional sense. For this purpose, using~\eqref{def.tau} and~\eqref{decompo.patuPeps}, for $(x,t) \in (x_i,x_{i+1})\times (t_n, t_{n + 1})$, we introduce the following functions 
\[
a_m(x) := \frac{x - x_i}{\Delta x_m}, \quad b_m(x) := 1 - \frac{x - x_i}{\Delta x_m}, \quad \lambda[u^p_M](x,t) := \left(\sigma_M^P(\cdot) \ast u^P_M(\cdot,t)\right)(x),
\]
and we have
{\small
\begin{equation}\label{decompo.conv}\begin{array}{ll}
\partial_t u^{P, m}_M(x,t) &=  a_m(x)\bigg\{ \lambda^M_{i + 1}[u^{P, n+ 1}]_+ \left(\theta_{i +3/2}^{P, n} + L^P\right) - \lambda^M_{i + 1}[u^{P, n+ 1}]_- \left(\theta_{i +1/2}^{P, n} + L^P\right)\bigg \} \\
&+ b_m(x) \bigg\{ \lambda^M_i[u^{P, n+ 1}]_+ \left(\theta_{i +1/2} ^{P, n} + L^P\right) - \lambda^M_i[u^{P, n+ 1}]_- \left(\theta_{i -1/2}^{P , n} + L^P\right)\bigg \} \\
&= \lambda[u^P_M]_+(x,t)\bigg\{ a_m(x)\left(\theta_{i +3/2}^{P, n} + L^P\right) + b_m(x) \left(\theta_{i +1/2}^{P, n} + L^P\right) \bigg\}\\
&- \lambda[u^P_M]_-(x,t)\bigg\{ a_m(x) \left(\theta_{i +1/2}^{P, n} + L^P\right) + b_m(x) \left(\theta_{i - 1/2}^{P, n} + L^P\right)\bigg\}+ e_m(x,t),
\end{array}\end{equation}}
 where
 {\footnotesize
\begin{align*}
e_m(x,t) &= a_m(x) \Bigg\{\bigg[\lambda^M_{i + 1}[u^{P, n+ 1}]_+ - \lambda[u^P_M]_+(x,t)\bigg]\left(\theta_{i +3/2}^{P,n} + L^P\right)- \bigg[\lambda^M_{i + 1}[u^{P, n+ 1}]_- - \lambda[u^P_M]_-(x,t)\bigg]\left(\theta_{i +1/2}^{P, n} + L^P\right)\bigg \} \\
&+ b_m(x) \Bigg\{ \bigg[\lambda^M_{i}[u^{P, n+ 1}]_+ - \lambda[u^P_M]_+(x,t)\bigg] \left(\theta_{i +1/2}^{P,n} + L^P\right)- \bigg[\lambda^M_i[u^{P, n+ 1}]_- - \lambda[u^P_M]_(x,t)-\bigg]\left(\theta_{i -1/2}^{P ,n} + L^P\right)\Bigg \}.
\end{align*}}
Now, let $\varphi \in C^\infty_c(I_P\times [0,T])$, using \eqref{TVestim} we have 
 \begin{multline}\label{e.eps}
       \Bigg| \int_{I_P\times[0,T]} \varphi(x,t)\, e_m(x,t) \, \dd x \dd t\Bigg| \leq 4T\|\varphi\|_{L^\infty(I_P\times[0,T])} \left(TV(u^P_0)\,e^{10 (L^P)T \, \|\K\|_{L^1(\R)}} +2PL^P\right)\\ \times \sup_{(y,\tau) \,\in\, \mathrm{supp}(\varphi)} \left(\sup_{|x_i-y|\leq \Delta x_m} \, \sup_{|t_{n+1}-\tau|\leq \Delta t_m} \left|\lambda^M_i[u^{P,n+1}]-\lambda[u^P_M](y,\tau)\right| \right).
 \end{multline}
Let $(y,\tau) \in \mathrm{supp}(\varphi)$ such that $|x_i-y|\leq \Delta x_m$ and $|t_{n+1}-\tau|\leq\Delta t_m$, we write
\begin{align*}
    \left|\lambda_i^M[u^{P,n+1}]-\lambda[u^P_M](y,\tau)\right|
    &= \left|\sum_{j\in\mathcal{I}_{N_m}} \int_{x_j}^{x_{j+1}} \sigma_M^P(x_j) u^{P,m}_M(x_i-x_j,t_{n+1}) \dd z - \int_{I_P} \sigma_M^P(z)  u^P_M(y-z,\tau)  \dd z\right|\\
    &\leq \sum_{j\in\mathcal{I}_{N_m}} \int_{x_j}^{x_{j+1}} \left|\left(\sigma_M^P(x_j)-\sigma_M^P(z)\right)\, u^{P,m}_M(x_i-x_j,t_{n+1}) \right| \, \dd z\\
    &+\sum_{j\in\mathcal{I}_{N_m}} \int_{x_j}^{x_{j+1}} \left|\sigma_M^P(z) \, \left(u^{P,m}_M(x_i-x_j,t_{n+1})-u^{P,m}_M(y-z,\tau)\right)\right|\,\dd z\\
    &+\sum_{j\in\mathcal{I}_{N_m}} \int_{x_j}^{x_{j+1}} \left|\sigma_M^P(z) \, \left(u^{P,m}_M(y-z,\tau)-u^P_M(y-z,\tau)\right)\right|\,\dd z\\
    &=: J_{m,6}+J_{m,7}+J_{m,8}.
\end{align*}
For $J_{m,6}$, using the $L^\infty$ bound~\eqref{Linftybounddis} we have 
\begin{align*}
    J_{m,6} &\leq \|u^P_0\|_{L^\infty(I_P)} \, e^{10 \, (L^P)T \, \|\K\|_{L^1(\R)}} \sum_{j\in\mathcal{I}_{N_m}} \int_{x_j}^{x_{j+1}} \left|\sigma_M^P(x_j)-\sigma_M^P(z)\right|\, \dd z \\
    &\leq \|u^P_0\|_{L^\infty(I_P)} \, e^{10 \, (L^P)T \, \|\K\|_{L^1(\R)}} \, \|\pa_x \sigma_M^P\|_{L^\infty(I_P)} \sum_{j\in\mathcal{I}_{N_m}}\, \int_{x_j}^{x_{j+1}} \left|x_j-z\right|\, \dd z\\
    &\leq P \, \Delta x_m \,\|u^P_0\|_{L^\infty(I_P)} \, e^{10 \, (L^P)T \, \|\K\|_{L^1(\R)}} \, \|\pa_x \sigma_M^P\|_{L^\infty(I_P)}. 
\end{align*}
Thus, using the relation $\sigma_M^P= (F_M \ast \K^P_M)/2P$, Young's convolution inequality and Lemma~\ref{kpdelta less than k}, it holds
\begin{align}\label{J6}
    J_{m,6} \leq  \frac{5 \,\Delta x_m}{2} \, \|u^P_0\|_{L^\infty(I_P)} \, e^{10 \, (L^P)T \, \|\K\|_{L^1(\R)}} \, \|\pa_x F_M\|_{L^\infty(I_P)} \, \|\K\|_{L^1(\R)} \to 0, \quad \mbox{as }m\to +\infty.
\end{align}
Now, for $J_{m,7}$, applying estimate~\eqref{mod.cont.dis} and Young's convolution inequality
we get 
\begin{align*}
    J_{m,7} &\leq 6 \, {C}_{\rm GE} \left(\frac{1}{\ln(1+\frac{1}{2\Delta x_m})}+\frac{1}{\ln(1+\frac{1}{\Delta t_m})}\right) \, \|\sigma_M^P\|_{L^1(I_P)}\\
    &\leq \frac{6 \, {C}_{\rm GE}}{2P} \left(\frac{1}{\ln(1+\frac{1}{2\Delta x_m})}+\frac{1}{\ln(1+\frac{1}{\Delta t_m})}\right) \, \|F_M\|_{L^1(I_P)} \, \|\K^P_M\|_{L^1(I_P)}\\
    &\leq {30 \, {C}_{\rm GE} } \left(\frac{1}{\ln(1+\frac{1}{2\Delta x_m})}+\frac{1}{\ln(1+\frac{1}{\Delta t_m})}\right) \, \|\K\|_{L^1(\R)}, 
\end{align*}
where we have used in the last inequality
 Lemma~\ref{kpdelta less than k} and the fact that $\frac{\|F_M\|_{L^1(I_P)}}{2P}=1$ Hence,
\begin{align}\label{J7}
    J_{m,7} \to 0, \quad \mbox{as }m\to+\infty.
\end{align}
Similarly, we can verify that 
\begin{align}\label{J8}
    J_{m,8} \le  \|u_M^{P,m}-u_M^{P} \|_{L^\infty(I_P)} \|\sigma_M^P\|_{L^1(I_P)} 
    \le 5 \|\K\|_{L^1(\R)} \|u_M^{P,m}-u_M^{P} \|_{L^\infty(I_P)}
    \to 0, \quad \mbox{as }m\to +\infty. 
\end{align} 
Therefore, gathering~\eqref{e.eps}--\eqref{J8} we deduce that
\begin{align*}
e_m \rightarrow 0, \quad \mbox{in } \mathcal{D}'(I_P\times (0, T)).
\end{align*}
Let us now introduce the function
\begin{align*}
 \theta^{P, m}_M(x,t) = \theta_{i +1/2}^{P , n}, \quad \mbox{for } (x,t) \in [x_i, x_{i + 1}) \times [t_n, t_{n+ 1}).
 \end{align*}
 Then, as the proof of estimate \eqref{estimationLlogLGE},  applying in particular estimate \eqref{gradentdis} and Dunford-Pettis Theorem (see Proposition \ref{weak star convergence}), we deduce the existence of $\theta^P_M$ such that for any $\varphi \in C_c(I_P\times(0,T))$ we have
 \begin{align}\label{weakL1.thetam}
 \int_{I_P \times (0, T)} \theta^{P, m}_M(x,t)  \, \varphi(x,t) \, \dd x \dd t \rightarrow \int_{I_P\times (0, T)} \theta^P_M(x,t) \, \varphi(x,t) \, \dd x \dd t, \quad \mbox{as }m \to+\infty.
\end{align}
Besides, for $x \in [x_i,x_{i+1})$ for some $i\in\mathcal{I}_{N_m}$, we rewrite the functions $a_m$ and $b_m$ as
\[
a_m(x) = \frac{x}{\Delta x_m} - \left\lfloor\frac{x}{\Delta x_m}\right\rfloor, \quad  b_m(x) = 1 - a_m(x),
\]
where $\lfloor \cdot \rfloor$ is the floor function, 
so that, equality~\eqref{decompo.conv} becomes
{\small
\begin{align*}
    \partial_t u^{P,m}_M(x,t) - e_m(x, t)&=  \lambda[u^P_M]_+(x,t)\left\{ a_m(x)\left(\theta^{P, m}_M(x + \Delta x_m,t) + L^P\right) + b_m(x) \left(\theta^{P, m}_M(x,t) + L^P\right) \right\}\\ 
    &- \lambda[u^P_M]_- (x,t)\left\{ a_m(x) \left(\theta^{P,m}_M(x,t) + L^P\right)  + b_m(x) \left(\theta^{P, m}_M(x - \Delta x_m,t) + L^P\right)\right\} .
\end{align*}}
Now we define, for any $\varphi \in C^\infty_c(I_P\times (0,T))$, the term $A_m$ by
\[
A_m = \int_{I_P\times(0, T) } \left(\partial_t u^{P, m}_M - e_m\right)(x,t) \, \varphi(x,t) \, \dd x \dd t.
\]
Recombining the terms, we rewrite $A_m$ as
{\small \begin{multline*}
    A_m = \int_{I_P\times (0, T)}\theta^{P,m}_M(x,t) \, \Big\{ a_m(x)\,\left(\lambda[u^P_M]_+ \varphi\right)(x - \Delta x_m,t) + b_m(x) \left(\lambda[u^P_M]_+ \varphi\right)(x,t) \\ -  a_m(x) \left(\lambda[u^P_M]_- \varphi\right)(x,t) - b_m(x) \left(\lambda[u^P_M]_- \varphi\right)(x + \Delta x_m,t)\Big\} \, \dd x \dd t + L^P \int_{I_P \times (0,T)} (\lambda[u^P_M] \varphi)(x,t)  \, \dd x \dd t.
\end{multline*} }
Now, define $B_m$ by
{\small\begin{align*}
  B_m &= \int_{I_P\times (0, T)}\theta^{P,m}_M(x,t) \Big\{ a_m(x)\left(\lambda[u^P_M]_+ \varphi\right)(x,t)+ b_m(x) \left(\lambda[u^P_M]_+\varphi\right)(x,t)\\
  &\phantom{xxxxxx}-  a_m(x) \left(\lambda[u^P_M]_- \varphi\right)(x,t)  - b_m(x) \left(\lambda[u^P_M]_-\varphi\right)(x,t) \Big\}\,\dd x\dd t + L^P \int_{I_P\times(0,T)} (\lambda[u^P_M] \varphi)(x,t) \, \dd x \dd t,
\end{align*}}
which yields, using the relation $\lambda[u^P_M]=(\sigma_M^P\ast u^P_M)$,
\begin{align*}
  B_m = \int_{I_P\times (0,T)} \left(\sigma_M^P(\cdot) \ast u^P_M(\cdot,t)\right)(x) \, \varphi(x,t) \, \left(\theta^{P,m}_M(x,t)+L^P\right)\, \dd x \dd t.
\end{align*}
Therefore, since $(\sigma_M^P \ast u^P_M)$ is continuous and thanks to~\eqref{weakL1.thetam}, we have
\begin{align*}
    B_m \to \int_{I_P\times(0,T)} \left(\sigma_M(\cdot) \ast u^P_M(\cdot,t)\right)(x) \, \varphi(x,t) \,\left( \theta^P_M(x,t)+L^P\right) \, \dd x \dd t, \quad \mbox{as }m \to +\infty.
\end{align*}
Furthermore, we notice that
\begin{multline*}
    |A_m - B_m| \leq \sup_{\pm} \bigg\{ \| (\lambda[u^P_M]_\pm \varphi)(\cdot \mp \Delta x_m,\cdot) - (\lambda[u^P_M]_\pm \varphi)(\cdot, \cdot)\|_{L^{\infty}(I_P\times(0, T)} \bigg\} \,\|\theta^{P,m}_M\|_{L^1(I_P\times(0,T))}.% \longrightarrow 0.
\end{multline*}
Besides, for all $(x,t)\in I_P\times (0,T)$, we notice that it holds
\begin{align*}
    \big| (\lambda[u^P_M]_+ \varphi)(x - \Delta x_m,t)& - (\lambda[u^P_M]_+ \varphi)(x, t)\big|\\
    &\leq \|\varphi\|_{L^\infty(I_P\times(0,T))} \int_{I_P} \left|\sigma_M^P(y) \left(u^P_M(x-\Delta x_m-y,t)-u^P_M(x-y,t)\right)\right|\,\dd y\\
    &+ \|u^P_M\|_{L^\infty(I_P\times(0,T))} \int_{I_P} \left|\sigma_M^P(y) \left(\varphi(x,t)-\varphi(x-\Delta x_m,t)\right)\right| \, \dd y.
\end{align*}
Thereby, thanks to the regularity of $u^P_M$ and $\varphi$, we obtain
\begin{align*}
    \| (\lambda[u^P_M]_+ \varphi)(\cdot - \Delta x_m,\cdot)& - (\lambda[u^P_M]_+ \varphi)(\cdot, \cdot)\|_{L^\infty(I_P\times(0,T))} \to 0, \quad \mbox{as }m\to+\infty.
\end{align*}
Using similar arguments for the other terms appearing in the estimate of $A_m-B_m$, and the uniform (w.r.t.~$\eps_m$) $L^1$ bound on $\theta^{P,m}_M$ (cf. \eqref{TVestim}), we deduce that
\begin{align*}
    |A_m-B_m| \to 0, \quad \mbox{as }m \to +\infty.
\end{align*}
Hence, we conclude that
 \begin{equation*}
    \partial_t u^P_M = \left(\sigma_M^P \ast u^P_M\right) \, \left(\theta^P_M + L^P\right) \quad \mbox{in } \ \mathcal{D}'(I_P \times (0, T)).
 \end{equation*}
 Moreover, using equality~\eqref{uepsilonx} and arguing as before, we deduce 
 \[
 \partial_x u^{P, m}_M\rightarrow \theta^P_M \quad \mbox{in } \mathcal{D}'(I_P\times (0, T)) \mbox{ as }m\to +\infty,
 \]
 so that $\theta^P_M = \partial_x u^P_M$ in the sense of distributions. Finally, we conclude that
 \[
 \partial_t u^P_M = \left(\sigma_M^P\ast u^P_M\right) \, \left(\pa_x u^P_M + L^P\right) \quad  \mbox{in } \mathcal{D}'(I_P\times(0, T)),
 \]
 which finishes the proof of Theorem~\ref{convergence in periodic}.
    
\section{Numerical experiments}\label{sec.numexp}
We present in this section some numerical simulations that demonstrate the evolution of the solution $v$ of the system  (\ref{non periodic equation}), from its approximation $u^P_M+L^Px$, which is calculated numerically using the scheme \eqref{schemeIC}-\eqref{numerical compacted scheme}, where we used definition \eqref{cesaro} to calculate  $\sigma_{M,j}^P$. To do this, we consider the following initial data
$$v_0(x)= \frac{2}{\pi}\arctan(x)+1, 
$$
and a physical kernel that naturally appears in the so-called Peierls-Nabarro model, which is known in the framework of isotropic elasticity with edge dislocations (see \cite{311, 19} for more details). This kernel is given by 
$$K(x) =  \frac{\mu b^2}{2\pi (1-\nu)} \frac{x^2-\zeta^2}{(x^2+\zeta^2)^2}, 
$$
where $\nu=\frac{\lambda}{2(\lambda+\mu)}$  is the Poisson ratio and $\lambda, \mu>0$ are the Lamé coefficients for isotropic elasticity. Here, $\zeta \neq 0$  is a physical parameter depending only on the material and represents the size of the dislocation core. Moreover, the vector 
$\vec{b}=b(1,0)$ is the Burgers vector, which reflects the direction of dislocation motion. 
In our simulations, we took the special case $\frac{\mu b^2}{2\pi (1-\nu)} =1$ and $\zeta=1$ for simplicity. Obviously, these initial data and kernel verify the required assumptions {\bf (H1)-(H2)-(H2)'}. \\
Taking the parameters indicated in Table \ref{tab1}, we observe in Figure \ref{figure1} the evolution of $v\approx u^P_M+L^Px$ over time, as well as its convergence towards the stationary state  $v=0$ of the equation (\ref{non periodic equation}), which is reached at $T=1400$.  
Additionally, Figure \ref{figure2} shows the evolution of  the dislocation density 
$\partial_x v \approx \partial_x u^P_M+L^P$ with respect to time.  
We see in these simulations what is physically expected: the dislocations move along the Burgers vector $\vec{b}=b(1,0)$ until they reach the boundary of the domain, after which they disappear as the domain size becomes infinite. 
The total variation over time of the numerical solution, represented in Figure \ref{figure3}, also confirms the convergence towards the stationary state, starting from time $T=1400$. 
  \begin{figure}[ht!]
    \centering
    \begin{subfigure}[t]{0.45\textwidth}
        \centering
      \includegraphics[width=\textwidth]{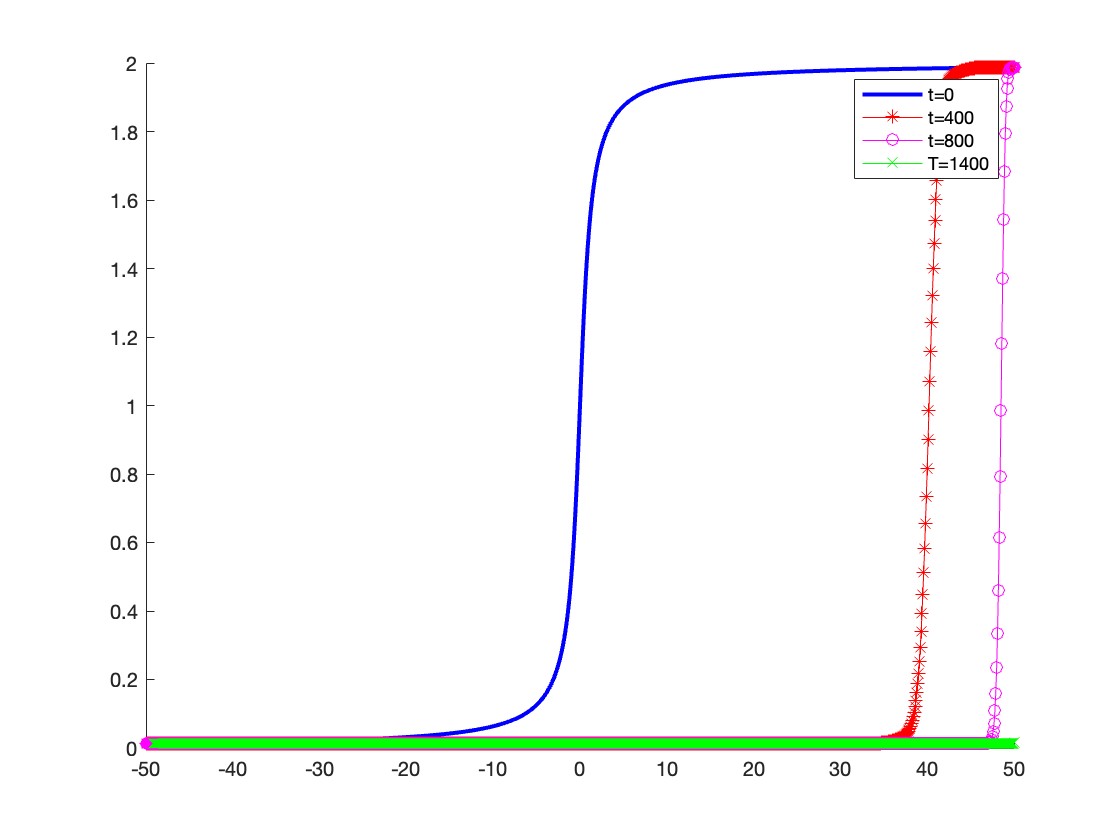}
        \caption{Approximate solution $u^P_M+L^Px$.}
        \label{figure1}
    \end{subfigure}%
    \hfill  \hspace{-2cm}
    \begin{subfigure}[t]{0.45\textwidth}
        %\centering
 \includegraphics[width=\textwidth]{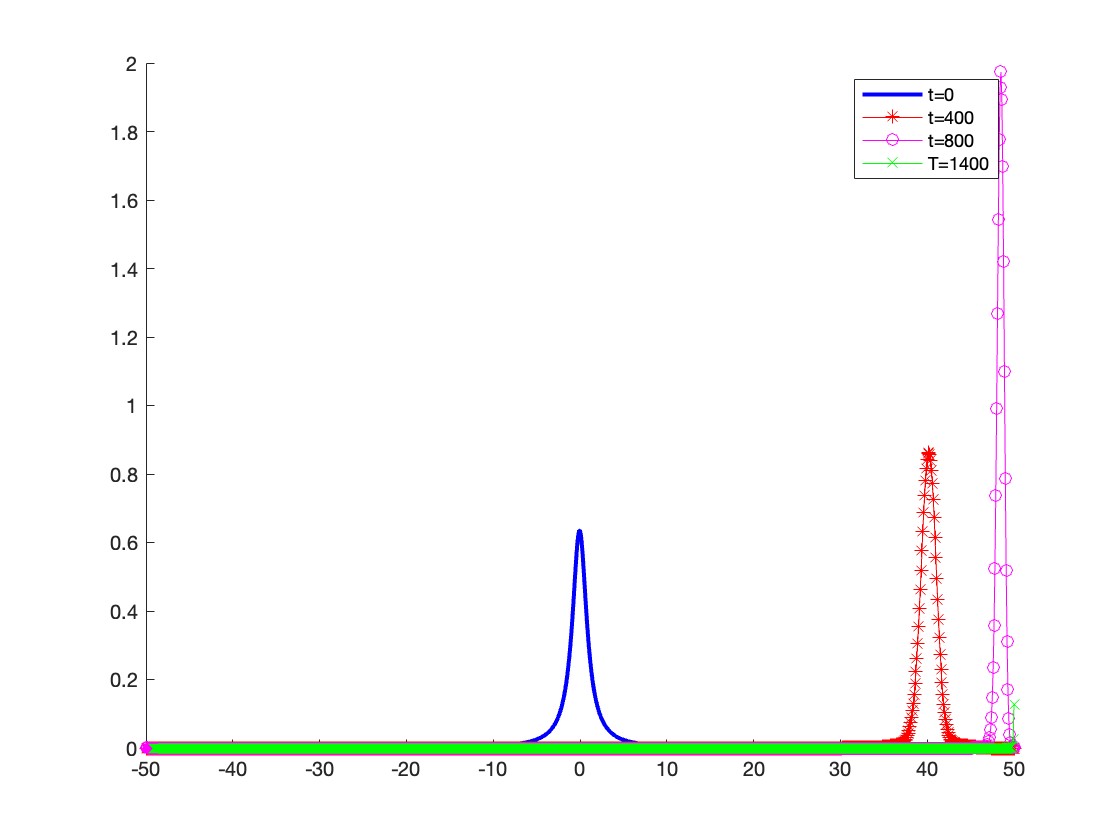}
      \caption{Approximate density $\partial_x u^P_M+L^P$.}
        \label{figure2}
         \end{subfigure} 
         \vspace{-0.2cm}
         \caption{Evolution towards stationary state.}
  \end{figure}     
\begin{center}
 \begin{tabular}{|*{4}{c|}}
     \hline
    \quad $M$ \quad &  \quad $P$  \quad &  \quad $N$  \quad &  \quad $N_t$  \\ \hline 
      \quad $400$  \quad &  \quad $50$  \quad &  \quad $500$  \quad & \quad $7\times 10^4$  \quad\\ \hline
   \end{tabular}
        \captionof{table}{Used parameters in  Figures \ref{figure1} and \ref{figure2}.}
        \label{tab1}  
   \end{center}
  \newpage 

It is important to note that the loss of periodicity in the solution in the above simulation is due to the linear term $L^Px$, which dampens the dislocation near the boundary and becomes increasingly weaker as $P$ increases.

  \begin{figure}[ht!]
        \centering
        \includegraphics[width=0.45\textwidth]{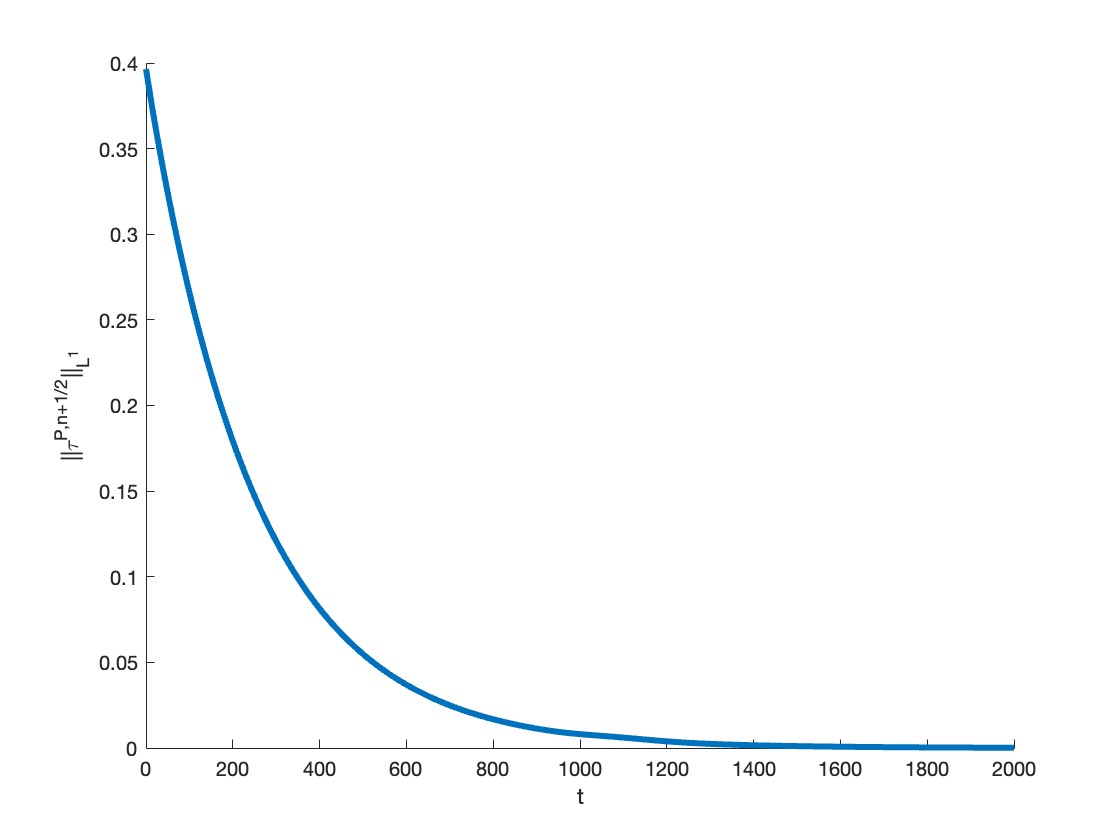}
        \caption{Total variation with respect to time of $u^P_M+L^Px.$}
        \label{figure3}
\end{figure}

In Figure  \ref{figure4}, we show the effect of the parameter $P$ on the precision of the numerical scheme. In particular, we note the impact of this parameter on the calculation of the solution $v\approx \partial_x u^P_M+L^P$ at $T=1400$. As expected, the results improve as $P$ increases, given the convergence of $u^P_M+L^Px$  to $v$ when $P \to +\infty$. The parameters used in the last simulation are specified in the following 
 table, which thus approves the decrease of the $L^\infty$-norm (at $T=1400$) with respect to this parameter.\\
 
 \begin{center}
 \begin{tabular}{|*{2}{c|}}
     \hline
 $M=100$,  $\Delta t=0.02$,   $\Delta x=0.1$   &  $L^\infty$-norm of $u^P_M+L^Px$  \\ \hline \hline
        $P=10$ &    $0.0635$  \quad\\ \hline
        $P=20$ &    $0.0319$  \quad\\ \hline
          $P=30$ &    $0.0212$  \quad\\ \hline
             $P=40$ &    $0.0159$  \quad\\ \hline
             $P=50$ &    $0.0127$  \quad\\ \hline
            $P=100$ &   $0.0065$ \quad\\ \hline
   \end{tabular}
        \captionof{table}{Used parameters in Figure \ref{figure4}.}
        \label{tab2}
 \end{center}

  \begin{figure}[ht!]
        \centering
        \includegraphics[width=0.50\textwidth]{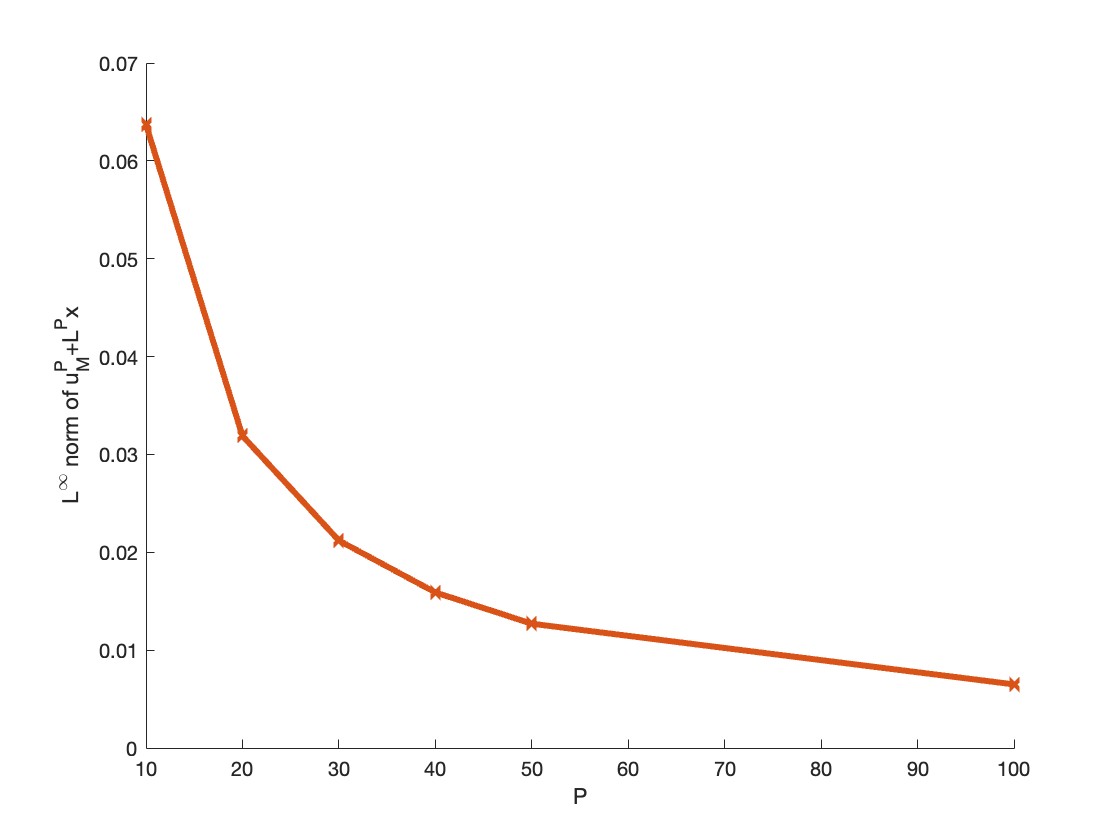}
        \caption{Impact of the parameter $P$.}
        \label{figure4}
\end{figure}

\begin{remark}\label{rem.alter-nu}
Although the proposed scheme provides better estimations than its alternative  \eqref{def.alternative}, no significant difference was observed in the numerical results between the two.
The only notable difference lies in computational speed: alternative  \eqref{def.alternative} is slightly faster, as it only requires the computation of a discrete integral, unlike definition \eqref{cesaro}, which involves a discrete convolution.
\end{remark}

\begin{appendix}

\section{Technical results}\label{app.techni}

In this section, we state technical results needed for our approach. In the sequel, we will denote by $I$ an interval of $\R$ with $I \subseteq \R$. Of course these results also hold if $I=I_P$ where we recall definition~\eqref{defI} of $I_P$.

\begin{lemma}[$ L\log L $ estimate]\label{Llog L estimate} 
If $w \in L^1(I)$ is a nonnegative function, then $ \int_{I} f(w) \dd x <  \infty$ if and only if $w \in L\log L(I)$. Moreover, we have the following estimates
\begin{equation}\label{f in llogl}
    \int_{I} f(w) \, \dd x \leq 1 + \|w\|_{L\,\log \,L(I)} + \|w\|_{L^1(I)} \, \ln\left(1 + \|w\|_{L\,\log \,L(I)}\right),
\end{equation}
and
\begin{equation}\label{w in llogl}
    \|w\|_{L\,\log \,L (I)}\leq 1 + \|w\|_{L^1(I)}\,\ln(1 + e^2) +   \int_{I} f(w) \, \dd x.
\end{equation}
\end{lemma}

\begin{proof}
See~\cite[Lemma 3.2]{HaMo10} and~\cite[Lemma 4.1]{2}.
\end{proof}

\begin{lemma}[Modulus of continuity]\label{modulus of continuity} 
Let $T > 0$, and assume that $w \in L^{\infty}(I\times(0,T))$ such that
\[
\|\partial_x w\|_{L^{\infty}(0, T; L\,\log\, L(I))} + \|\partial_t w\|_{L^{\infty}(0, T; L\,\log\, L(I))}  \leq C_{\rm mod}.
\]
Then, for all $\gamma, \,h \geq 0$ such that $x, x+\gamma \in I$ and $t \times (0, T - h)$, we have 
\[
\left|w(x + \gamma, t+h) - w(x, t)\right| \leq 6 C_{\rm mod} \, \left( \frac{1}{\ln(1 + \frac{1}{\gamma})} + \frac{1}{\ln(1 + \frac{1}{h})}\right).
\]
\end{lemma}

\begin{proof}
See~\cite[Lemma 4.3]{HaMo10} and~\cite[Lemma 4.4]{2}.
\end{proof}

\begin{proposition}[Weak-$L^1$ compactness] \label{weak star convergence} 
Consider a sequence of functions $(\theta^{\eta})_{\eta >0}$ satisfying for some $T>0$
\[
\|\theta^{\eta}(t)\|_{L^1(I)} + \int_{I} f\left(\theta^{\eta}(x,t)\right) \, \dd x \leq A_T, \quad \mbox{for a.e. } t \in (0, T),
\]
where $A_T$ is independent of $\eta$. Then, there exists a function $\theta$ and a constant $C_T=C_T(A_T)$ such that 
\[
\|\theta(t)\|_{L^1(I)} + \int_{I} f\left(\theta(x,t)\right) \, \dd x \leq C_T, \quad \mbox{for a.e. } t \in (0, T).
\]
Moreover, for any function $\varphi \in C_c (I\times(0,T))$, we have 
\[
\int_{I\times (0,T)} \theta^{\eta} \, \varphi \, \dd x \dd t \rightarrow \int_{I \times (0, T)} \theta \, \varphi \, \dd x \dd t \quad \mbox{as }\eta \to 0.
\]
\end{proposition}

\begin{proof}
See~\cite[Proposition 4.6]{2} and references therein.
\end{proof}

\section{Proof of Proposition~\ref{prop.PMinfty}}\label{app.proofPMinfty}

Thanks to Theorem~\ref{convergence in periodic}  we know that the constants appearing in~\eqref{second property}--\eqref{mod.cont.period} do not depend on $M_\ell$ and $P_k$. Therefore, since for any $k\in\N$ and for any $(x,t)\in I_{P_k} \times (0,T-h)$ there exists a constant $C>0$ independent of $M_\ell$ and $P_k$ such that
\begin{align*}
    \left|u_{k,\ell}(x+\gamma,t+h)-u_{k,\ell}(x,t)\right| \leq C \left(\frac{1}{\ln(1+\frac{1}{\gamma})} + \frac{1}{\ln(1+\frac{1}{h})}\right), \quad \forall \ell \in \N,
\end{align*}
we deduce from Arzel\`a-Ascoli theorem, that we can extract a converging subsequence still denoted $(u_{k,\ell})_{\ell\in \N}$ which converges uniformly to a function $u_k$ as $\ell \to+ \infty$ (for any $k\in\N$). In particular, $u_k(\cdot,0)=u^{P_k}_0(\cdot)$ and $u_k$ satisfies
\begin{align}\label{aux}
    \left|u_k(x+\gamma,t+h)-u_k(x,t)\right| \leq C \left(\frac{1}{\ln(1+\frac{1}{\gamma})} + \frac{1}{\ln(1+\frac{1}{h})}\right), \quad \forall k \in \N.
\end{align}
Arguing as in the proof of Theorem~\ref{convergence in periodic} one can easily show that, as $\ell \to \infty$,
\begin{align}\label{weakstar.l.to.infty}
    \pa_x u_{k,\ell} \rightharpoonup \pa_x u_k \quad \mbox{weakly} \mbox{ in } L^1(I_{P_k}\times(0,T)), \quad \forall k\in\N.
\end{align}
Furthermore, since the constant in~\eqref{aux} does not depend on $P_k$, then applying once more time Arzel\`a-Ascoli theorem we deduce the existence of a converging subsequence still denoted $(u_k)_{k\in \N}$ which converges uniformly to a limit function $v$ on every compact set $K \subset \R \times [0,T)$. In particular, $v(\cdot,0)=v_0(\cdot)$ and $v$ also satisfies
\begin{align*}
    \left|v(x+\gamma,t+h)-v(x,t)\right| \leq C \left(\frac{1}{\ln(1+\frac{1}{\gamma})} + \frac{1}{\ln(1+\frac{1}{h})}\right). 
\end{align*}
Moreover, we have, when $k \to +\infty$,
\begin{align}\label{weakstar.k.to.infty}
    \pa_x u_k \rightharpoonup \pa_x v \quad \mbox{weakly}\mbox{ in } L^1_{\rm loc}(\R \times(0,T)).
\end{align}
Let us now identify $v$ as a solution to~\eqref{non periodic equation} in the distributional sense. For this purpose, let $\varphi \in C^\infty_c(\R \times (0,T))$, and assume that $k$ is large enough such that $\mathrm{supp}(\varphi) \subset (-P_k,P_k)\times(0,T)$. We consider the term:
\begin{align*}
    \int_{\R \times(0,T)} (\sigma_{M_\ell}^{P_k}(\cdot)\ast u_{k,\ell}(\cdot,t)(x)) \, \left(\pa_x u_{k,\ell}(x,t) + L^{P_k}\right) \varphi \, \dd x \dd t.
\end{align*}
We first notice, thanks to Young's convolution inequality, Theorem~\ref{convergence in periodic}, Lemma~\ref{lem.init.perio.to.nonperiod} and Lemma~\ref{kpdelta less than k}, that it holds

\begin{align*}
    \Bigg| \int_{\R \times(0,T)}(\sigma_{M_\ell}^{P_k}(\cdot)\ast u_{k,\ell}(\cdot,t)(x)) \, L^{P_k}& \varphi(x,t) \, \dd x \dd t \Bigg| \leq L^{P_k} \, \|\sigma_{M_\ell}^{P_k}\|_{L^1(I_{P_k})} \, \|u_{k,\ell}\|_{L^\infty(I_{P_k})} \, \|\varphi\|_{L^1(\R)}\\
    &\leq \frac{2}{P_k} \|v_0\|_{L^\infty(\R)}^2 \, \|\varphi\|_{L^1(\R)} \, \|\sigma_{M_\ell}^{P_k}\|_{L^1(I_{P_k})} \, e^{10 \|v_0\|_{L^\infty(\R)} T \, \|\K\|_{L^1(\R)}} \\
    &\leq  \frac{2}{P_k} \|v_0\|_{L^\infty(\R)}^2 \, \|\varphi\|_{L^1(\R)}\, \|\K^{P_k}_{M_\ell}\|_{L^1(I_{P_k})} \, e^{10 \|v_0\|_{L^\infty(\R)} T \, \|\K\|_{L^1(\R)}}\\
    &\leq  \frac{10}{P_k} \|v_0\|_{L^\infty(\R)}^2 \, \|\varphi\|_{L^1(\R)} \, \|\K\|_{L^1(\R)} \, e^{10\|v_0\|_{L^\infty(\R)} T \, \|\K\|_{L^1(\R)}}.
\end{align*}
Hence, passing to the limit $\ell \to+\infty$ and $k\to+\infty$ leads to
\begin{align}
    \int_{\R \times(0,T)} (\sigma_{M_\ell}^{P_k}(\cdot)\ast u_{k,\ell}(\cdot,t)(x)) \, L^{P_k} \varphi(x,t)\, \dd x \dd t \to 0.
\end{align}
Now, we consider the term
\begin{align*}
    \int_{\R \times(0,T)}(\sigma_{M_\ell}^{P_k}(\cdot)\ast u_{k,\ell}(\cdot,t)(x))  \, \pa_x u_{k,\ell}(x,t) \, \varphi(x,t) \, \dd x \dd t.
\end{align*}
We notice, thanks to Fej\'er's theorem and the uniform convergence (up to a subsequence) of $(u_{k,\ell})_{\ell\in \N}$ toward $u_k$, that for any fixed $k \in \N$, it holds
\begin{align*}
    (\sigma_{M_\ell}^{P_k} \ast u_{k,\ell}) \to (\K^{P_k} \ast u_k) - {4u_k} \left(\int_{|x|\geq P_k} \left|\K(x)\right| \, \dd x\right) \quad \mbox{in } L^\infty(I_{P_k}\times(0,T)), \, \mbox{as }\ell \to +\infty.
\end{align*}
Thus, applying~\eqref{weakstar.l.to.infty} yields, as $\ell \to +\infty$,
\begin{multline*}
    \int_{\R \times(0,T)} (\sigma_{M_\ell}^{P_k}  \ast u_{k,\ell}) \, \pa_x u_{k,\ell} \, \varphi \, \dd x \dd t \to \int_{\R\times(0,T)} (\K^{P_k}\ast u_k) \, \pa_x u_k\, \varphi \, \dd x \dd t\\
    - {4} \left(\int_{|x|\geq P_k} |\K(x)|\,\dd x\right) \left(\int_{\R \times (0,T)} u_k \, \pa_x u_k \, \varphi \, \dd x \dd t\right).
\end{multline*}
We readily deduce that the second term in the right hand side converges to zero as $k \to +\infty$, 
since $\K\in L^1(\R)$ and $u_k$  is uniformly bounded in $L^\infty(\R\times (0,T))$ as well as
$\pa_x u_k$  is uniformly bounded in $L^\infty((0,T); L^1(\R))$. Moreover, for the first term, we notice, thanks to the periodicity of $\K^{P_k}$ and $u_k$, that for a.e. $(x,t) \in \R\times(0,T)$ we have the following decomposition:
\begin{align*}
    (\K^{P_k}\ast u_k(t))(x) &- (\K \star v(t))(x) = \int_{-P_k}^{P_k} \K^{P_k}(x-y) \, u_k(y,t) \, \dd y - \int_\R \K(x-y) \, v(y,t) \, \dd y\\
    &= \int_\R \mathds{1}_{[-P_k,P_k]}(y) \,\K(x-y) \left(u_k(y,t)-v(y,t)\right) \, \dd y - \int_{|y|\geq P_k} \K(x-y) \, v(y,t) \, \dd y.
\end{align*}
Now, we notice, as a consequence of Theorem~\ref{convergence in periodic}, that we have
\begin{align*}
  \|v\|_{L^\infty(\R\times(0,T))} \leq \|v_0\|_{L^\infty(\R)}.
\end{align*}
Moreover, since $(u_k)_{k\in\N}$ converges locally uniformly (up to a subsequence) toward $v$ as $k \to +\infty$, we deduce that $u_k \to v$ strongly in $L^\infty(\R \times (0,T))$. Therefore,
\begin{align*}
    \left|(\K^{P_k}\ast u_k)(x) - (\K \star v)(x)\right| \leq \|u_k-v\|_{L^\infty(\R \times (0,T))} \, \|\K\|_{L^1(\R)} + \|v_0\|_{L^\infty(\R)} \int_{|y|\geq P_k} \left|\K(y)\right| \, \dd y,
\end{align*}
so that
\begin{align*}
    (\K^{P_k} \ast u_k) \to (\K\star v) \quad \mbox{strongly in }L^\infty(\R\times(0,T)), \, \mbox{as }k \to +\infty.
\end{align*}
Finally, using~\eqref{weakstar.k.to.infty}, we deduce that
\begin{align*}
    \int_{\R\times(0,T)} (\K^{P_k}\ast u_k) \, \pa_x u_k\, \varphi \, \dd x \dd t \to \int_{\R \times (0,T)} (\K \star v) \, \pa_x v \, \varphi\, \dd x \dd t, \quad \mbox{as }k\to+\infty.
\end{align*}
Thereby, $v$ satisfies
\begin{align*}
    \pa_t v = (\K\star v) \, \pa_x v \quad \mbox{in }\mathcal{D}'(\R \times (0,T)),
\end{align*}
which concludes the proof of Proposition~\ref{prop.PMinfty}.

\end{appendix}

\bibliographystyle{plain}
\bibliography{bibliography}

\begin{thebibliography}{10}

\bibitem{111}
R.~A. Adams.
\newblock {\em Sobolev spaces}, volume~65 of {\em Pure Appl. Math., Academic
  Press}.
\newblock Academic Press, New York, NY, 1975.

\bibitem{211}
O.~Alvarez, P.~Cardaliaguet, and R.~Monneau.
\newblock Existence and uniqueness for dislocation dynamics with nonnegative
  velocity.
\newblock {\em Interfaces Free Bound.}, 7(4):415--434, 2005.

\bibitem{250}
O.~Alvarez, E.~Carlini, R.~Monneau, and E.~Rouy.
\newblock Convergence of a first order scheme for a non-local eikonal equation.
\newblock {\em Applied Numerical Mathematics}, 56(9):1136--1146, 2006.
\newblock Numerical Methods for Viscosity Solutions and Applications.

\bibitem{1420}
O.~Alvarez, P.~Hoch, Y.~Le~Bouar, and R.~Monneau.
\newblock Small-time solvability of a non-local {Hamilton}-{Jacobi} equation
  describing dislocation dynamics.
\newblock {\em C. R., Math., Acad. Sci. Paris}, 338(9):679--684, 2004.

\bibitem{311}
O.~Alvarez, P.~Hoch, Y.~Le~Bouar, and R.~Monneau.
\newblock Dislocation dynamics: {Short}-time existence and uniqueness of the
  solution.
\newblock {\em Arch. Ration. Mech. Anal.}, 181(3):449--504, 2006.

\bibitem{19}
P.~M. Anderson, J.~P. Hirth, and J.~Lothe.
\newblock {\em Theory of dislocations}.
\newblock Cambridge: Cambridge University Press, 3rd edition edition, 2017.

\bibitem{5}
G.~Barles.
\newblock {\em Solutions de viscosit{\'e} des {\'e}quations de
  {Hamilton}-{Jacobi}}, volume~17 of {\em Math. Appl. (Berl.)}.
\newblock Paris: Springer-Verlag, 1994.

\bibitem{6}
G.~Barles.
\newblock A new stability result for viscosity solutions of nonlinear parabolic
  equations with weak convergence in time.
\newblock {\em C. R., Math., Acad. Sci. Paris}, 343(3):173--178, 2006.

\bibitem{7}
G.~Barles, P.~Cardaliaguet, O.~Ley, and R.~Monneau.
\newblock Global existence results and uniqueness for dislocation equations.
\newblock {\em SIAM J. Math. Anal.}, 40(1):44--69, 2008.

\bibitem{8}
G.~Barles, P.~Cardaliaguet, O.~Ley, and A.~Monteillet.
\newblock Uniqueness results for nonlocal {Hamilton}-{Jacobi} equations.
\newblock {\em J. Funct. Anal.}, 257(5):1261--1287, 2009.

\bibitem{1720}
G.~Barles and O.~Ley.
\newblock Nonlocal first-order {Hamilton}-{Jacobi} equations modelling
  dislocations dynamics.
\newblock {\em Commun. Partial Differ. Equations}, 31(8):1191--1208, 2006.

\bibitem{10}
R.~Boudjerada and A.~El~Hajj.
\newblock Global existence results for eikonal equation with {{\(BV\)}} initial
  data.
\newblock {\em NoDEA, Nonlinear Differ. Equ. Appl.}, 22(4):947--978, 2015.

\bibitem{11}
R.~Boudjerada, A.~El~Hajj, and M.~S. Moulay.
\newblock Existence result for a one-dimensional eikonal equation.
\newblock {\em C. R., Math., Acad. Sci. Paris}, 353(2):133--137, 2015.

\bibitem{13}
E.~Carlini, N.~Forcadel, and R.~Monneau.
\newblock A generalized fast marching method for dislocation dynamics.
\newblock {\em SIAM J. Numer. Anal.}, 49(6):2470--2500, 2011.

\bibitem{14}
M.~G. Crandall, H.~Ishii, and P.-L. Lions.
\newblock User's guide to viscosity solutions of second order partial
  differential equations.
\newblock {\em Bull. Am. Math. Soc., New Ser.}, 27(1):1--67, 1992.

\bibitem{15}
M.~G. Crandall and P.-L. Lions.
\newblock Viscosity solutions of {Hamilton}-{Jacobi} equations.
\newblock {\em Trans. Am. Math. Soc.}, 277:1--42, 1983.

\bibitem{1350}
M.~G. Crandall and P.-L. Lions.
\newblock Two approximations of solutions of {Hamilton}-{Jacobi} equations.
\newblock {\em Math. Comput.}, 43:1--19, 1984.

\bibitem{1}
A.~El~Hajj.
\newblock Global solution for a non-local eikonal equation modelling
  dislocation dynamics.
\newblock {\em Nonlinear Anal., Theory Methods Appl., Ser. A, Theory Methods},
  168:154--175, 2018.

\bibitem{HaMo10}
A.~El~Hajj and R.~Monneau.
\newblock Global continuous solutions for diagonal hyperbolic systems with
  large and monotone data.
\newblock {\em J. Hyperbolic Differ. Equ.}, 7(1):139--164, 2010.

\bibitem{3000}
N.~Forcadel.
\newblock Dislocation dynamics with a mean curvature term: short time existence
  and uniqueness.
\newblock {\em Differ. Integral Equ.}, 21(3-4):285--304, 2008.

\bibitem{17}
N.~Forcadel.
\newblock Comparison principle for a generalized fast marching method.
\newblock {\em SIAM J. Numer. Anal.}, 47(3):1923--1951, 2009.

\bibitem{18}
N.~Forcadel, C.~Le~Guyader, and C.~Gout.
\newblock Generalized fast marching method: applications to image segmentation.
\newblock {\em Numer. Algorithms}, 48(1-3):189--211, 2008.

\bibitem{50000}
A.~Ghorbel and R.~Monneau.
\newblock Well-posedness and numerical analysis of a one-dimensional non-local
  transport equation modelling dislocations dynamics.
\newblock {\em Math. Comput.}, 79(271):1535--1564, 2010.

\bibitem{4}
K.~Hoffman.
\newblock {\em Banach spaces of analytic functions.}
\newblock New York: Dover Publications, Inc., reprint of the 1962 original
  edition, 1988.

\bibitem{2}
L.~Monasse and R.~Monneau.
\newblock Gradient entropy estimate and convergence of a semi-explicit scheme
  for diagonal hyperbolic systems.
\newblock {\em SIAM J. Numer. Anal.}, 52(6):2792--2814, 2014.

\bibitem{1620}
S.~Osher and J.~A. Sethian.
\newblock Fronts propagating with curvature-dependent speed: {Algorithms} based
  on {Hamilton}-{Jacobi} formulations.
\newblock {\em J. Comput. Phys.}, 79(1):12--49, 1988.

\bibitem{22}
D.~Rodney, Y.~{Le Bouar}, and A.~Finel.
\newblock Phase field methods and dislocations.
\newblock {\em Acta Materialia}, 51(1):17--30, 2003.

-

\end{thebibliography}

\end{document}